\documentclass[12pt]{amsart}
\usepackage{amssymb}
\usepackage{amsbsy}
\usepackage{amscd}
\usepackage{enumerate}
\usepackage{mathrsfs}
\usepackage{eucal}
\usepackage{mathtools}
\usepackage[hidelinks,unicode]{hyperref}

\numberwithin{equation}{section}
\theoremstyle{plain}
\newtheorem{theorem}{Theorem}[section]
\newtheorem{lemma}[theorem]{Lemma}
\newtheorem{proposition}[theorem]{Proposition}
\newtheorem{corollary}[theorem]{Corollary}
\theoremstyle{definition}
\newtheorem{definition}[theorem]{Definition}
\newtheorem*{definition*}{Definition}
\theoremstyle{remark}

\newcommand{\Y}{Y_{\hbar}}
\newcommand{\g}{\mathfrak g}
\newcommand{\HH}[1]{\widetilde h_{#1,1}}
\newcommand{\xp}[2]{x^+_{#1,#2}}
\newcommand{\xm}[2]{x^-_{#1,#2}}

\DeclareUnicodeCharacter{2013}{--}
\DeclareUnicodeCharacter{2014}{---}
\DeclareUnicodeCharacter{2019}{'}

\title{Affine Super Yangian and Weyl Groupoid}
\author{Vasiliy Volkov}
\address{Center of Fundamental Mathematics, Moscow Institute of Physics and Technology, Dolgoprudny, Russia}
\email{volkov.vd@phystech.edu}
\author{Vladimir Stukopin}
\address{Moscow Institute of Physics and Technology, Dolgoprudny, Russia; South Mathematical Institute, Vladikavkaz, Russia}
\email{stukopin.va@phystech.edu}
\date{}
\hypersetup{pdftitle={Affine Super Yangian and Weyl Groupoid},pdfauthor={Vasiliy Volkov and Vladimir Stukopin}}

\begin{document}
\begin{abstract}
We study affine super Yangians of type $\widehat{\mathfrak{sl}}(m|n)$ over $\mathbb C[\hbar]$ for arbitrary cyclic parity sequences, with $m,n\geq2$ and $m\neq n$. We prove the equivalence of the full current and degree-zero-and-one presentations by reconstructing all current relations. A finite-type RTT calculation and a cut-diagram argument give explicit normalized quantum odd reflections, including the affine vertex. Together with even braid lifts and diagonal sign automorphisms, these maps satisfy the defining relations of an extended reflection groupoid. We compare the positive-root tensors under odd reflection and identify their difference with the reduced coproduct of the quadratic Cartan correction. Starting from the distinguished coproduct, this establishes the prescribed coassociative coproduct for every cyclic parity sequence. We specify the compatible degreewise completions needed for coproduct comparison at the affine vertex.
\end{abstract}

\maketitle
\tableofcontents
\section{Introduction}\label{sec:introduction}
Yangians were introduced by Drinfeld as Hopf algebra quantizations of polynomial current Lie bialgebras, closely related to rational solutions of the quantum Yang--Baxter equation \cite{Drinfeld85}. His subsequent current presentation \cite{Drinfeld88} became a basic tool in their study. Levendorskii established PBW bases and finite systems of generators and defining relations \cite{LevendorskiiPBW,LevendorskiiRelations}. Classical accounts of Yangians, their representations, and the surrounding theory of quantum groups are provided by the monographs of Molev and Chari--Pressley \cite{Molev,ChariPressley}.

The affine theory began with the construction of Boyarchenko and Levendorskii \cite{BoyarchenkoLevendorskii}; later, Guay, Nakajima, and Wendlandt developed minimal presentations and completed coproducts for affine Yangians \cite{GNW}. Early superalgebraic developments include Nazarov's quantum Berezinian \cite{Nazarov91} and Stukopin's Yangians of Lie superalgebras of type $A(m,n)$ \cite{Stukopin94}. Serganova's work on generalized root systems \cite{Serganova96} provides a framework for isotropic roots and odd reflections, which are central to the dependence of super Yangian presentations on the choice of simple roots.

The connection with integrable systems is expressed through rational $R$-matrices and commuting transfer matrices. Isaev's \emph{Lectures on Quantum Groups and Yang--Baxter Equations} \cite{IsaevLectures} give a detailed account of this approach and its superalgebraic variants. Fuksa, Isaev, Karakhanyan, and Kirschner \cite{FIKK} construct Yang--Baxter $R$-operators, evaluations, and fusion for orthosymplectic Yangians. Another connection concerns $W$-algebras: Brundan and Kleshchev describe finite $W$-algebras of type $A$ using shifted Yangians \cite{BrundanKleshchev}, while Poletaeva and Serganova obtain a Yangian description of a principal finite $W$-superalgebra in the queer case \cite{PoletaevaSerganova}. In the affine super setting, Ueda constructs homomorphisms to enveloping algebras of rectangular $W$-superalgebras \cite{UedaW}.

Guay, Nakajima, and Wendlandt construct coproducts for ordinary affine Yangians \cite{GNW}; Kodera studies braid operators \cite{Kodera}. Ueda constructs the distinguished affine super Yangian and its completed coproduct \cite{Ueda}. Related minimal presentations and Hopf structures are studied in \cite{MS}. Equivalence of full and minimal affine presentations and reflection isomorphisms were stated in \cite{VS}; the Weyl groupoid and coproducts for varying simple-root systems were considered further in \cite{VS24} and \cite{VS25}. Here we give direct proofs over $\mathbb C[\hbar]$, fix the reflection normalizations, and compare the prescribed positive-root coproducts. The finite-type Drinfeld--RTT theorem for arbitrary parity \cite{Tsymbaliuk} supplies the local input.

One of the results of this work is the construction of a Yangian Weyl groupoid or Yangian reflection groupoid. Essentially, we extend the action of the Weyl groupoid to super-Yangians by constructing an action of odd reflections, requiring these to be bialgebra homomorphisms. In effect, we deform the action of the standard Weyl groupoid in tandem with the deformation that yields the super-Yangian.

Let us briefly outline the main results of this paper.

Theorem~\ref{thm:presentation-equivalence} proves equivalence of the full and minimal presentations. It reconstructs all current relations from a finite list using independence of proper subsets of columns of the symmetric root matrix.

Theorem~\ref{thm:quantum-odd-reflections} gives normalized odd reflections on the degree-zero and degree-one generators, including the affine vertex. The finite RTT calculation fixes the root-vector signs and spectral shift. Each relation check involves at most four cyclic vertices and takes place in a finite path obtained by deleting one vertex. This verifies the minimal relations, including the additional Serre relations, and the inverse identities. Theorems~\ref{thm:normalized-coxeter-relations} and \ref{thm:original-braid-lifts} prove the mixed braid relations and identify the squares of the original even lifts; Corollary~\ref{cor:extended-groupoid-action} establishes the action of the extended reflection groupoid defined at the beginning of Section~\ref{sec:the-extended-reflection-groupoid}.

For the coproduct, we compute the finite change of the half Casimir tensor and identify it with the reduced coproduct of the quadratic Cartan correction (Theorem~\ref{thm:coproduct-reflection-comparison}). Starting from Ueda's distinguished coproduct, this comparison proves the stated root-sum formula for every cyclic parity sequence (Theorem~\ref{thm:coproduct-all-parities}).

The degreewise completion is specified before the root sums are used. At the affine vertex the reflection is compatible with the physical loop grading inherited from the source. The compatible topology also permits comparison at the affine vertex and tensor products of modules on which the completed coproducts act.

The companion manuscript \cite{PBW}, \emph{PBW theorem for affine super Yangian}, treats the distinguished specialization at $\hbar=1$.

\section{Conventions and cyclic root data}\label{sec:conventions-and-cyclic-root-data}
All algebras and vector spaces are over $\mathbb C$ unless a coefficient ring is specified. For homogeneous elements, we use
\begin{equation}
[a,b]=ab-(-1)^{p(a)p(b)}ba,\qquad
\{a,b\}=ab+(-1)^{p(a)p(b)}ba.
\label{eq:super-brackets}
\end{equation}
In particular, the supercommutator of two odd elements is symmetric, whereas their super-anticommutator is their ordinary commutator. Tensor products have multiplication
\begin{equation}
(a\otimes b)(c\otimes d)=(-1)^{p(b)p(c)}ac\otimes bd,\quad
\sigma(a\otimes b)=(-1)^{p(a)p(b)}b\otimes a.
\label{eq:graded-tensor-product}
\end{equation}
The super-Jacobi identity is used in the derivation form
\[
[a,[b,c]]=[[a,b],c]+(-1)^{p(a)p(b)}[b,[a,c]].
\]
Fix integers $m,n\geq2$, $m\neq n$, and put $N=m+n\geq5$. Choose an ordered homogeneous basis with parity sequence
\[
(p_1,\ldots,p_N)\in\{0,1\}^N,\qquad
\#\{a:p_a=0\}=m,\qquad\#\{a:p_a=1\}=n.
\]
Extend the parities periodically, and let $w_a$ be weight symbols with $(w_a,w_b)=(-1)^{p_a}\delta_{ab}$. Adjoin the null root $\delta$, orthogonal to the $w_a$, and set
\begin{equation}
\alpha_i=w_i-w_{i+1}\quad(1\leq i<N),\qquad
\alpha_0=\delta+w_N-w_1.
\label{eq:cyclic-simple-roots}
\end{equation}
The index set is $I=\mathbb Z/N\mathbb Z$, represented by $0,1,\ldots,N-1$; the index $0$ is also represented by $N$ when cyclic formulas involve signs. Set $p(\alpha_i)=p_i+p_{i+1}$ (mod $2$), with this cyclic convention, and
\begin{equation}
B=(b_{ij})_{i,j\in I},\qquad b_{ij}=(\alpha_i,\alpha_j).
\label{eq:symmetric-root-matrix}
\end{equation}
For cyclic indices,
\begin{equation}
b_{ii}=(-1)^{p_i}+(-1)^{p_{i+1}},\quad
b_{i,i+1}=-(-1)^{p_{i+1}},\quad b_{i,i-1}=-(-1)^{p_i},
\label{eq:cyclic-root-matrix-entries}
\end{equation}
and all other entries vanish. Thus $b_{ii}=0$ precisely at the odd simple roots; each adjacent off-diagonal entry is nonzero and belongs to $\{1,-1\}$.

\section{The two presentations}\label{sec:the-two-presentations}
Throughout this section, $N=m+n\geq5$, $m,n\geq2$, and $m\neq n$. We use the cyclic index set $I=\mathbb Z/N\mathbb Z$, the signed cyclic matrix $B=(b_{ij})$, and the root parity $p(\alpha_i)$ introduced above. In particular, $b_{ii}\in\{0,2,-2\}$; $b_{ii}=0$ if and only if $p_i\neq p_{i+1}$, and $b_{i,i\pm1}\in\{1,-1\}$. All brackets and anticommutators below are interpreted as the super bracket and super anticommutator. The coefficient ring is $\mathbb C[\hbar]$.

\begin{definition}\label{def:full-current-presentation}
Let $Y^D_\hbar(\Pi)$ be the unital associative $\mathbb C[\hbar]$-superalgebra with generators $h_{i,r}$ and $x^\pm_{i,r}$, where $i\in I$ and $r\geq0$. Their parities are $p(h_{i,r})=0$ and $p(x^\pm_{i,r})=p(\alpha_i)$. Its defining relations are
\begin{align}
[h_{i,r},h_{j,s}]&=0,\label{eq:full-cartan-commutators}\\
[h_{i,0},x^\pm_{j,r}]&=\pm b_{ij}x^\pm_{j,r},\label{eq:full-cartan-action}\\
[x^+_{i,r},x^-_{j,s}]&=\delta_{ij}h_{i,r+s},\label{eq:full-mixed-root-bracket}\\
[h_{i,r+1},x^\pm_{j,s}]-[h_{i,r},x^\pm_{j,s+1}]
&=\pm\frac\hbar2b_{ij}\{h_{i,r},x^\pm_{j,s}\},\label{eq:full-cartan-current-shift}\\
[x^\pm_{i,r+1},x^\pm_{j,s}]-[x^\pm_{i,r},x^\pm_{j,s+1}]
&=\pm\frac\hbar2b_{ij}\{x^\pm_{i,r},x^\pm_{j,s}\},\label{eq:full-root-current-shift}
\end{align}
for all admissible indices. For $i\neq j$, the Serre relations are
\begin{equation}
\sum_{\sigma\in S_{1+|b_{ij}|}}
[x^\pm_{i,r_{\sigma(1)}},[x^\pm_{i,r_{\sigma(2)}},\ldots,
[x^\pm_{i,r_{\sigma(1+|b_{ij}|)}},x^\pm_{j,s}]\ldots]]=0.
\label{eq:full-serre-relations}
\end{equation}
For each odd vertex $i$, we also impose
\begin{align}
[x^\pm_{i,r},x^\pm_{i,s}]&=0,\label{eq:full-isotropic-relations}\\
[[x^\pm_{i-1,r},x^\pm_{i,0}],[x^\pm_{i,0},x^\pm_{i+1,s}]]&=0.\label{eq:full-quartic-serre-relations}
\end{align}
All current indices in these relations are nonnegative integers.
\end{definition}

\begin{definition}\label{def:minimal-presentation}
Let $Y^{\min}_\hbar(\Pi)$ be generated by $h_{i,0},h_{i,1},x^\pm_{i,0},x^\pm_{i,1}$, with the same parities as above. Write
\begin{equation}
\widetilde h_{i,1}=h_{i,1}-\frac\hbar2 h_{i,0}^2.\label{eq:shifted-cartan-generator}
\end{equation}
The defining relations, for $a,b\in\{0,1\}$ where applicable, are
\begin{align}
[h_{i,a},h_{j,b}]&=0,\label{eq:minimal-cartan-commutators}\\
[x^+_{i,0},x^-_{j,0}]&=\delta_{ij}h_{i,0},\label{eq:minimal-degree-zero-root-bracket}\\
[x^+_{i,1},x^-_{j,0}]&=\delta_{ij}h_{i,1}=[x^+_{i,0},x^-_{j,1}],\label{eq:minimal-degree-one-root-bracket}\\
[h_{i,0},x^\pm_{j,a}]&=\pm b_{ij}x^\pm_{j,a},\label{eq:minimal-cartan-action}\\
[\widetilde h_{i,1},x^\pm_{j,0}]&=\pm b_{ij}x^\pm_{j,1},\label{eq:minimal-shifted-cartan-action}\\
[x^\pm_{i,1},x^\pm_{j,0}]-[x^\pm_{i,0},x^\pm_{j,1}]
&=\pm\frac\hbar2b_{ij}\{x^\pm_{i,0},x^\pm_{j,0}\},\label{eq:minimal-root-current-shift}\\
(\operatorname{ad}x^\pm_{i,0})^{1+|b_{ij}|}(x^\pm_{j,0})&=0\quad(i\neq j),\label{eq:minimal-serre-relations}\\
[x^\pm_{i,0},x^\pm_{i,0}]&=0\quad(p(\alpha_i)=1),\label{eq:minimal-isotropic-relations}\\
[[x^\pm_{i-1,0},x^\pm_{i,0}],[x^\pm_{i,0},x^\pm_{i+1,0}]]&=0\quad(p(\alpha_i)=1).\label{eq:minimal-quartic-serre-relations}
\end{align}
The word ``minimal'' denotes the customary degree-zero-and-one presentation.
\end{definition}

\begin{proposition}\label{prop:degree-one-generation}
The algebra $Y^D_\hbar(\Pi)$ is generated by its generators of degrees zero and one. Sending the generators in Definition~\ref{def:minimal-presentation} to the equally named generators in Definition~\ref{def:full-current-presentation} defines a surjective homomorphism
\[
\pi:Y^{\min}_\hbar(\Pi)\longrightarrow Y^D_\hbar(\Pi).
\]
\end{proposition}

\begin{proof}
Equations \eqref{eq:full-cartan-action} and \eqref{eq:full-cartan-current-shift}, with the first current index equal to zero, give
\begin{equation}
[\widetilde h_{i,1},x^\pm_{j,r}]=\pm b_{ij}x^\pm_{j,r+1}.\label{eq:shifted-cartan-current-action}
\end{equation}
Indeed, $[h_{i,0}^2,x^\pm_{j,r}]=h_{i,0}[h_{i,0},x^\pm_{j,r}]+[h_{i,0},x^\pm_{j,r}]h_{i,0}$ cancels the anticommutator term in \eqref{eq:full-cartan-current-shift}. For each $i$, choose $\nu(i)$ with $b_{\nu(i),i}\neq0$; one may choose $\nu(i)=i$ at an even vertex and $\nu(i)=i+1$ at an odd vertex. Then
\begin{equation}
x^\pm_{i,r+1}=\pm\frac1{b_{\nu(i),i}}[\widetilde h_{\nu(i),1},x^\pm_{i,r}],\qquad
h_{i,r+1}=[x^+_{i,r+1},x^-_{i,0}].\label{eq:current-generator-recursion}
\end{equation}
This proves generation. Every relation of the degree-zero-and-one presentation
follows from the full presentation; for \eqref{eq:minimal-serre-relations}, specialize
all repeated current indices in \eqref{eq:full-serre-relations} to zero and divide
by $(1+|b_{ij}|)!$. Consequently the indicated homomorphism is well defined
and surjective.\end{proof}

\par\medskip
We henceforth define the higher elements $x^{\pm}_{i,r},h_{i,r}$
\textit{inside} $Y^{\min}_{\hbar}(\Pi)$ by \eqref{eq:current-generator-recursion}.
The formulas at $r=0$ agree with the original generators by
\eqref{eq:minimal-shifted-cartan-action} and \eqref{eq:minimal-degree-one-root-bracket}.
All subsequent reconstruction statements concern these elements of
$Y^{\min}_{\hbar}(\Pi)$, rather than relations already imposed in
$Y^{\mathrm D}_{\hbar}(\Pi)$.

\section{Equivalence of the presentations}\label{sec:equivalence-of-the-presentations}

\subsection{Rank and shift operators}\label{subsec:rank-and-shift-operators}

\begin{lemma}\label{lem:cyclic-matrix-kernel}
The kernel of $B$ is $\mathbb C(1,\ldots,1)$.
Consequently every proper subset of its columns is linearly independent.
In particular, any two distinct columns, and any three distinct columns,
are independent.
\end{lemma}

\begin{proof} Let $S=\operatorname{diag}((-1)^{p_a})$ be the signed metric,
with $\sum_a(-1)^{p_a}=m-n$. Define the cyclic difference operator by
$(Dv)_a=v_a-v_{a-1}$. Then $B=D^TSD$. If $Bv=0$, the vector $SDv$
belongs to $\ker D^T=\mathbb C(1,\ldots,1)$, so
$Dv=\lambda((-1)^{p_a})_a$. Summing its entries gives $0=\lambda(m-n)$;
hence $\lambda=0$ and $v$ is constant. A linear relation supported on a
proper subset of columns would give a vector in $\ker B$ with a zero entry,
and therefore would be the zero relation. Since $N\geq5$, the asserted
two- and three-column cases are proper subsets.\end{proof}

\begin{lemma}\label{lem:cartan-shift-action}
For every $i,j\in I$ and $r\geq0$, the elements constructed in
$Y^{\min}_{\hbar}(\Pi)$ satisfy
\begin{align}
  [h_{i,0},x^{\pm}_{j,r}]&=\pm b_{ij}x^{\pm}_{j,r},\label{eq:reconstructed-cartan-action}\\
  [\widetilde h_{i,1},x^{\pm}_{j,r}]&=\pm b_{ij}x^{\pm}_{j,r+1}.
\label{eq:reconstructed-shift-action}
\end{align}
In particular, $[h_{i,0},h_{j,r}]=0$ for all $r$.
\end{lemma}

\begin{proof} The elements $h_{i,0}$ and $\widetilde h_{i,1}$ commute
with all $\widetilde h_{k,1}$ by \eqref{eq:minimal-cartan-commutators}.
The identities at $r=0$ are defining relations. Commuting either $h_{i,0}$
or $\widetilde h_{i,1}$ past the outer bracket in
\eqref{eq:current-generator-recursion} proves the induction step by the ordinary
Jacobi identity for an even outer element. The weights of $x^+_{j,r}$
and $x^-_{j,0}$ cancel in their bracket, which proves the last assertion.
\end{proof}

\begin{lemma}\label{lem:opposite-anti-involution}
There is a parity-preserving, $\mathbb C[\hbar]$-linear
anti-involution $\tau$ of $Y^{\min}_{\hbar}(\Pi)$ satisfying
\[
  \tau(ab)=\tau(b)\tau(a),\qquad
  \tau(x^{\pm}_{i,a})=x^{\mp}_{i,a},\qquad
  \tau(h_{i,a})=h_{i,a}\quad(a=0,1).
\]
For every $r$, it satisfies $\tau(x^{\pm}_{i,r})=x^{\mp}_{i,r}$ and
\[
  \tau(h_{i,r})=[x^+_{i,0},x^-_{i,r}].
\]
\end{lemma}

Its equality with $h_{i,r}$ is established in
Lemma~\ref{lem:mixed-root-relations}.

\begin{proof} For homogeneous $u,v$ the ordinary reversal rule gives
\[
\begin{aligned}
  \tau([u,v])&=-(-1)^{p(u)p(v)}[\tau(u),\tau(v)],
  &\quad\tau(\{u,v\})&=(-1)^{p(u)p(v)}\{\tau(u),\tau(v)\}.
\end{aligned}
\]
These formulas send each relation
\eqref{eq:minimal-cartan-commutators}--\eqref{eq:minimal-quartic-serre-relations} to a scalar
multiple of another relation in the same list.
In \eqref{eq:minimal-degree-one-root-bracket}, the two equalities are interchanged.
Thus the reversal descends from the free algebra and its square is the
identity. Since $\tau(\widetilde h_{i,1})=\widetilde h_{i,1}$, induction
in \eqref{eq:current-generator-recursion} gives the assertion for $x^{\pm}_{i,r}$.
Applying the displayed reversal formula to the definition of $h_{i,r}$
gives the last formula.\end{proof}

\begin{proposition}\label{prop:off-diagonal-current-relations}
For $i\ne j$ and all $r,s\ge0$, the reconstructed elements of $\Y^{\min}(\Pi)$ satisfy
\begin{align}
[\xp{i}{r},\xm{j}{s}]&=0,\label{eq:reconstructed-off-diagonal-mixed-bracket}\\
[x^\pm_{i,r+1},x^\pm_{j,s}]-[x^\pm_{i,r},x^\pm_{j,s+1}]&=\pm\frac{\hbar b_{ij}}2\{x^\pm_{i,r},x^\pm_{j,s}\}.\label{eq:reconstructed-root-current-shift}
\end{align}
If $b_{ij}=0$, then also $[x^\pm_{i,r},x^\pm_{j,s}]=0$.
\end{proposition}

\begin{proof} Put $C(r,s)=[\xp{i}{r},\xm{j}{s}]$. By \eqref{eq:reconstructed-shift-action},
\[
[\HH{k},C(r,s)]=b_{ki}C(r+1,s)-b_{kj}C(r,s+1).
\]
The two relevant columns of $B$ are independent by Lemma~\ref{lem:cyclic-matrix-kernel}. If $C(r,s)=0$, two suitable values of $k$ therefore imply $C(r+1,s)=C(r,s+1)=0$. Starting from \eqref{eq:minimal-degree-zero-root-bracket} proves \eqref{eq:reconstructed-off-diagonal-mixed-bracket}.

For a fixed sign $\epsilon\in\{1,-1\}$, let $F(r,s)$ be the difference between the two sides of \eqref{eq:reconstructed-root-current-shift}. The derivation $\operatorname{ad}\HH{k}$ acts on this difference by
\[
[\HH{k},F(r,s)]=\epsilon b_{ki}F(r+1,s)+\epsilon b_{kj}F(r,s+1).
\]
Its value at $(0,0)$ is zero by \eqref{eq:minimal-root-current-shift}; the same two-column induction proves all values zero. When $b_{ij}=0$, apply this argument instead to $[x^\epsilon_{i,r},x^\epsilon_{j,s}]$, whose initial value is zero by \eqref{eq:minimal-serre-relations}.\end{proof}

\begin{proposition}\label{prop:off-diagonal-cartan-root}
For $i\ne j$ and $r,s\ge0$, one has in $\Y^{\min}(\Pi)$
\begin{equation}
[h_{j,s+1},\xp{i}{r}]-[h_{j,s},\xp{i}{r+1}]=\frac{\hbar b_{ji}}2\{h_{j,s},\xp{i}{r}\}.\label{eq:reconstructed-cartan-current-shift}
\end{equation}
\end{proposition}

\begin{proof} Bracket \eqref{eq:reconstructed-root-current-shift}, with positive sign, on the right with $\xm{j}{0}$. The off-diagonal mixed brackets vanish by \eqref{eq:reconstructed-off-diagonal-mixed-bracket}; moreover $[\xp{j}{s},\xm{j}{0}]=h_{j,s}$ by definition. For the homogeneous elements $\xp{i}{r}$ and $\xp{j}{s}$, the remaining product calculation is
\[
[\{\xp{i}{r},\xp{j}{s}\},\xm{j}{0}]=\xp{i}{r}h_{j,s}+h_{j,s}\xp{i}{r}.
\]
This gives exactly \eqref{eq:reconstructed-cartan-current-shift}, since $B$ is symmetric.\end{proof}

\begin{lemma}\label{lem:quadratic-shift-operator}
Define in $\Y^{\min}(\Pi)$
\begin{equation}
L_{j,2}=h_{j,2}-\hbar h_{j,0}h_{j,1}+\frac{\hbar^2}{3}h_{j,0}^3.\label{eq:second-logarithmic-cartan-generator}
\end{equation}
For $i\ne j$ and every $r\ge0$,
\begin{equation}
[L_{j,2},\xp{i}{r}]=b_{ji}\xp{i}{r+2}+\frac{\hbar^2 b_{ji}^3}{12}\xp{i}{r}.\label{eq:second-logarithmic-shift-action}
\end{equation}
Consequently
\[
Q_{j;i,2}=L_{j,2}-\frac{\hbar^2 b_{ji}^2}{12}h_{j,0}
\quad\text{satisfies}\quad [Q_{j;i,2},\xp{i}{r}]=b_{ji}\xp{i}{r+2}.
\]
\end{lemma}

\begin{proof} The weight relation and Proposition~\ref{prop:off-diagonal-cartan-root} give
\begin{align*}
[h_{j,0},\xp{i}{r+a}]&=b_{ji}\xp{i}{r+a},\\
[h_{j,1},\xp{i}{r}]&=b_{ji}\xp{i}{r+1}+\frac{\hbar b_{ji}}2\{h_{j,0},\xp{i}{r}\},\\
[h_{j,2},\xp{i}{r}]&=b_{ji}\xp{i}{r+2}+\frac{\hbar b_{ji}}2\bigl(\{h_{j,0},\xp{i}{r+1}\}+\{h_{j,1},\xp{i}{r}\}\bigr).
\end{align*}
Since $[h_{j,0},h_{j,1}]=0$, substitution yields
\[
\begin{aligned}
[h_{j,2}-\hbar h_{j,0}h_{j,1},\xp{i}{r}]
&=b_{ji}\xp{i}{r+2}+\frac{\hbar^2 b_{ji}^2}{4}\{h_{j,0},\xp{i}{r}\}\\
&\quad-\frac{\hbar^2 b_{ji}}2h_{j,0}\{h_{j,0},\xp{i}{r}\}.
\end{aligned}
\]
Adding $\frac{\hbar^2 b_{ji}}3(h_{j,0}^2\xp{i}{r}+h_{j,0}\xp{i}{r}h_{j,0}+\xp{i}{r}h_{j,0}^2)$, the commutator of $\hbar^2h_{j,0}^3/3$ with $\xp{i}{r}$, and using $h_{j,0}\xp{i}{r}=\xp{i}{r}(h_{j,0}+b_{ji})$ reduces the right-hand side to $b_{ji}\xp{i}{r+2}+\hbar^2b_{ji}^3\xp{i}{r}/12$. The final assertion follows from $[h_{j,0},\xp{i}{r}]=b_{ji}\xp{i}{r}$.\end{proof}

\subsection{Propagation of the same-sign and Serre relations}\label{subsec:propagation-of-the-same-sign-and-serre-relations}
\begin{lemma}\label{lem:two-shift-propagation}
Suppose $P(r,s)=P(s,r)$ is a family of vectors, and two linear operators $D_1,D_2$ satisfy
\[
\begin{aligned}
D_1P(r,s)&=P(r+1,s)+P(r,s+1),\\
D_2P(r,s)&=P(r+2,s)+P(r,s+2).
\end{aligned}
\]
If $P(0,0)=0$, then $P(r,s)=0$ for all $r,s\ge0$.
\end{lemma}

\begin{proof} Symmetry and the first identity give $2P(1,0)=0$. Assume all terms of total index at most $q+1$ vanish and put $r+s=q$. Applying $D_1$ to $P(r+1,s)$ and $P(r,s+1)$ shows $P(r+2,s)=P(r,s+2)$. Applying $D_2$ to $P(r,s)$ gives their sum zero, so both vanish, and either first-shift equation then gives $P(r+1,s+1)=0$. These equations cover every pair of total index $q+2$. Induction proves the claim.\end{proof}

\begin{proposition}\label{prop:diagonal-same-sign-relations}
The relations \eqref{eq:full-isotropic-relations} and the diagonal cases of \eqref{eq:full-root-current-shift} hold for the reconstructed elements of $\Y^{\min}(\Pi)$.
\end{proposition}

\begin{proof} First consider positive generators. Fix $i$ and choose a neighbour $j\ne i$. The even derivations
\[
D_1=b_{ji}^{-1}\operatorname{ad}\HH{j},\qquad D_2=b_{ji}^{-1}\operatorname{ad}Q_{j;i,2}
\]
send $\xp{i}{r}$ to $\xp{i}{r+1}$ and $\xp{i}{r+2}$, respectively. If $i$ is odd, the family $P(r,s)=[\xp{i}{r},\xp{i}{s}]$ is symmetric, and its initial value is zero by \eqref{eq:minimal-isotropic-relations}. The two-shift induction therefore proves \eqref{eq:full-isotropic-relations}, which also implies \eqref{eq:full-root-current-shift} at this vertex since $b_{ii}=0$.

If $i$ is even, take instead
\[
P(r,s)=[\xp{i}{r+1},\xp{i}{s}]-[\xp{i}{r},\xp{i}{s+1}]-\frac{\hbar b_{ii}}2\{\xp{i}{r},\xp{i}{s}\}.
\]
Skew-symmetry of the even bracket makes this family symmetric in $r,s$. The Leibniz rule gives exactly the two shift identities of Lemma~\ref{lem:two-shift-propagation}, and \eqref{eq:minimal-root-current-shift} supplies $P(0,0)=0$. This proves the positive diagonal current relation. The negative statements follow by applying the anti-involution of Lemma~\ref{lem:opposite-anti-involution}.\end{proof}

\begin{proposition}\label{prop:serre-propagation}
The reconstructed generators of $\Y^{\min}(\Pi)$ satisfy \eqref{eq:full-serre-relations} for all current indices and all pairs $i\ne j$.
\end{proposition}

\begin{proof} The case $b_{ij}=0$ was proved in Proposition~\ref{prop:off-diagonal-current-relations}. Otherwise $i,j$ are adjacent and $q_{ij}=2$. Work first with positive generators and set
\[
P(r,t;s)=[\xp{i}{r},[\xp{i}{t},\xp{j}{s}]]+[\xp{i}{t},[\xp{i}{r},\xp{j}{s}]].
\]
By Lemma~\ref{lem:cyclic-matrix-kernel}, a linear combination $H$ of the $\HH{a}$ can be chosen with $[H,\xp{i}{r}]=0$ and $[H,\xp{j}{s}]=\xp{j}{s+1}$. The initial relation \eqref{eq:minimal-serre-relations} gives $P(0,0;0)=0$; successive applications of $\operatorname{ad}H$ give $P(0,0;s)=0$ for every $s$.

Let $k$ be the other neighbour of $i$, away from $j$. Since $N\ge5$, it is distinct from $i,j$, with $b_{ki}\ne0$ and $b_{kj}=0$. The operators
\[
D_1=b_{ki}^{-1}\operatorname{ad}\HH{k},\qquad
D_2=b_{ki}^{-1}\operatorname{ad}\left(L_{k,2}-\frac{\hbar^2}{12}h_{k,0}\right)
\]
annihilate every $\xp{j}{s}$ and shift $\xp{i}{r}$ by one and two current degrees, respectively. Here $b_{ki}^2=1$ and Lemma~\ref{lem:quadratic-shift-operator} also applies to the pair $k,j$ with $b_{kj}=0$. For fixed $s$, the family $P(r,t;s)$ is symmetric in $r,t$ and obeys the two-shift identities. Consequently it vanishes for every $r,t$ by Lemma~\ref{lem:two-shift-propagation}. The negative-sign relations follow by the anti-involution.\end{proof}

\begin{proposition}\label{prop:quartic-propagation}
For every odd vertex $i$, the reconstructed generators satisfy \eqref{eq:full-quartic-serre-relations} for all $r,s\ge0$.
\end{proposition}

\begin{proof} Fix either sign and write
\[
X(r,a,b,s)=[[x^\pm_{i-1,r},x^\pm_{i,a}],[x^\pm_{i,b},x^\pm_{i+1,s}]].
\]
Assume inductively that $X(r,0,0,s)=0$ for $r+s\le d$. For $r+s=d$, applying $\operatorname{ad}\HH{v}$ and dividing out the common sign gives, for every $v\in I$,
\[
\begin{aligned}
0={}&b_{v,i-1}X(r+1,0,0,s)\\
&+b_{v,i}\bigl(X(r,1,0,s)+X(r,0,1,s)\bigr)\\
&+b_{v,i+1}X(r,0,0,s+1).
\end{aligned}
\]
The three displayed columns of $B$ are independent by Lemma~\ref{lem:cyclic-matrix-kernel}. Some three rows therefore give an invertible coefficient matrix in the three parenthesized unknowns. In particular $X(r+1,0,0,s)=X(r,0,0,s+1)=0$. The base case is \eqref{eq:minimal-quartic-serre-relations}, so the induction is complete.\end{proof}

\subsection{Mixed and Cartan relations}\label{subsec:mixed-and-cartan-relations}
All calculations in this subsection take place in $\Y^{\min}(\Pi)$. The preceding propositions establish all same-sign and Serre relations, together with the off-diagonal mixed and positive Cartan--root relations. The following argument proves the remaining relations.

\begin{lemma}\label{lem:logarithmic-shift-operators}
Let $h_0,\ldots,h_d$ be even elements of an associative superalgebra over $\mathbb C[\hbar]$, and let $x_s$, $s\ge0$, have a fixed parity. Suppose
\[
[h_0,x_s]=\epsilon b x_s,\qquad
[h_{r+1},x_s]-[h_r,x_{s+1}]=\frac{\epsilon\hbar b}2\{h_r,x_s\}\quad(0\le r<d),
\]
where $\epsilon\in\{1,-1\}$ and $b\in\mathbb C$. Define
\[
L_n=[v^{-n-1}]\frac1\hbar\log\left(1+\hbar\sum_{r=0}^{d}h_rv^{-r-1}\right)\qquad(0\le n\le d).
\]
This expression is polynomial in $\hbar$ and the $h_r$; in particular, $L_n=h_n+P_n(h_0,\ldots,h_{n-1})$. The elements $h_r$ may be noncommuting. Then
\begin{equation}
[L_n,x_s]=\epsilon b\sum_{q=0}^{\lfloor n/2\rfloor}\binom n{2q}\frac{(\hbar b/2)^{2q}}{2q+1}x_{s+n-2q}.\label{eq:rank-one-logarithmic-action}
\end{equation}
Consequently there is a polynomial $Q_n^{(b)}=h_n+P_n^{(b)}(h_0,\ldots,h_{n-1})$ satisfying $[Q_n^{(b)},x_s]=\epsilon b x_{s+n}$. The same polynomial works for both signs whenever both families satisfy the hypotheses.
\end{lemma}

\begin{proof} Put $A(v)=\sum_{r=0}^{d}h_rv^{-r-1}$, $H(v)=1+\hbar A(v)$, and let $T$ denote the formal shift $x_s\mapsto x_{s+1}$. Define an operator
\[
U(a)=[A(v),a]H(v)^{-1}.
\]
Summing the assumed recurrences directly gives
\[
(v-T)[A(v),x_s]=\epsilon bx_s+\frac{\epsilon\hbar b}2\{A(v),x_s\}\quad(\bmod v^{-d-1}).
\]
Since $\{A,x\}=[A,x]+2xA$, right multiplication by $H^{-1}$ yields
\[
U(x_s)=\epsilon b(v-T-\epsilon\hbar b/2)^{-1}x_s\quad(\bmod v^{-d-2}).
\]
All rational expressions in $v-T$ are expanded at $v=\infty$; each coefficient encodes a finite linear combination of the $x_{s+a}$. The following universal formal identity has polynomial coefficients:
\[
\operatorname{ad}\left(\sum_{m\ge1}\frac{(-1)^{m-1}\hbar^{m-1}}m A^m\right)
=\sum_{k\ge1}\frac{(-1)^{k-1}\hbar^{k-1}}k U^k.
\]
For a direct verification, let $\ell$ and $\rho$ be left and right multiplication by $A$. They commute, and $U=(\ell-\rho)(1+\hbar\rho)^{-1}$. The asserted operator identity is the specialization of the scalar formal-power-series identity
\[
\sum_{m\ge1}\frac{(-1)^{m-1}\hbar^{m-1}}m(\ell^m-\rho^m)
=\sum_{k\ge1}\frac{(-1)^{k-1}\hbar^{k-1}}k\left(\frac{\ell-\rho}{1+\hbar\rho}\right)^k.
\]
Its coefficients are identities of polynomials over $\mathbb Q$ in commuting indeterminates. These polynomial identities specialize in the given $\mathbb C[\hbar]$-algebra. Substitution of the computed action of $U$ gives
\[
\operatorname{ad}\left(\frac{\log H(v)}\hbar\right)(x_s)
=\epsilon b\sum_{q\ge0}\frac{(\hbar b/2)^{2q}}{2q+1}(v-T)^{-2q-1}x_s\quad(\bmod v^{-d-2}).
\]
Taking the coefficient of $v^{-n-1}$ proves \eqref{eq:rank-one-logarithmic-action}. Finally define recursively
\[
Q_n^{(b)}=L_n-\sum_{q=1}^{\lfloor n/2\rfloor}\binom n{2q}\frac{(\hbar b/2)^{2q}}{2q+1}Q_{n-2q}^{(b)}.
\]
All operations take place in the original polynomial base ring: the notation $\hbar^{-1}\log H$ abbreviates a series whose coefficients are already divisible by $\hbar$ before division.\end{proof}
\par

\begin{lemma}\label{lem:mixed-root-relations}
The reconstructed elements of $\Y^{\min}(\Pi)$ satisfy
\[
[\xp{i}{r},\xm{i}{s}]=h_{i,r+s}\qquad(r,s\ge0).
\]
\end{lemma}

\begin{proof} Put $D_i(r,s)=[\xp{i}{r},\xm{i}{s}]$ and $h_{i,r}=D_i(r,0)$. The degree-zero and degree-one mixed relations are defining relations. Applying $\operatorname{ad}\HH{k}$ to $D_i(1,0)=D_i(0,1)=h_{i,1}$, and using $[\HH{k},h_{i,1}]=0$, gives $D_i(2,0)=D_i(1,1)=D_i(0,2)$; choose $k$ with $b_{ki}\ne0$. Thus the assertion holds through total degree two.

Assume it holds for total degree at most $d$, where $d\ge2$. The identities
\[
[\HH{k},\xp{i}{r}]=b_{ki}\xp{i}{r+1},\qquad
[\HH{k},\xm{i}{s}]=-b_{ki}\xm{i}{s+1}
\]
show that the element
\[
C_i=D_i(r+1,d-r)-D_i(r,d-r+1)
\]
is independent of $r$, $0\le r\le d$, and that
\begin{equation}
[\HH{k},h_{i,d}]=b_{ki}C_i.\label{eq:diagonal-mixed-bracket-defect-action}
\end{equation}
Indeed, apply the even derivation $\operatorname{ad}\HH{k}$ to $D_i(r,d-r)=h_{i,d}$ and first choose one $k$ with $b_{ki}\ne0$.

Fix $j\ne i$ with $b_{ij}\ne0$. Bracketing the established off-diagonal same-sign current relations with $\xm{i}{0}$ gives the Cartan--positive-root recurrence for the family $h_{i,r}=D_i(r,0)$. Bracketing the negative relations with $\xp{i}{0}$ gives its negative analogue for $D_i(0,r)$. The induction hypothesis identifies these two families for $r\le d$. Therefore Lemma~\ref{lem:logarithmic-shift-operators}, applied only through degree $d$, produces
\[
\begin{aligned}
Q_{i,d}^{(j)}&=h_{i,d}+P(h_{i,0},\ldots,h_{i,d-1}),\\
[Q_{i,d}^{(j)},\xp{j}{s}]&=b_{ij}\xp{j}{s+d},\qquad
[Q_{i,d}^{(j)},\xm{j}{s}]=-b_{ij}\xm{j}{s+d}.
\end{aligned}
\]
All lower $h_{i,r}$ commute with $h_{j,1}$: their commutators with $\HH{j}$ vanish by the induction hypothesis in total degree $r+1$, and their $h_{j,0}$-weights are zero. Hence
\[
[Q_{i,d}^{(j)},h_{j,1}]=[h_{i,d},h_{j,1}]=-b_{ij}C_i.
\]
On the other hand, applying $\operatorname{ad}Q_{i,d}^{(j)}$ to $h_{j,1}=D_j(1,0)$ yields
\[
[Q_{i,d}^{(j)},h_{j,1}]=b_{ij}\bigl(D_j(d+1,0)-D_j(1,d)\bigr)=db_{ij}C_j.
\]
Thus $C_i=-dC_j$. Interchanging $i$ and $j$ gives $C_j=-dC_i$, so $(d^2-1)C_i=0$. Since the base field has characteristic zero and $d\ge2$, $C_i=0$. Every vertex has a suitable neighbour; therefore all degree-$d+1$ mixed brackets on a fixed anti-diagonal coincide. Their common value is $D_i(d+1,0)=h_{i,d+1}$. This completes the induction.\end{proof}

\begin{corollary}\label{cor:cartan-shift-commutativity}
For all $i,k$ and $r\ge0$, $[\HH{k},h_{i,r}]=[h_{k,0},h_{i,r}]=0$.
\end{corollary}

\begin{proof} Apply the shift identity to $h_{i,r}=[\xp{i}{r},\xm{i}{0}]$ and use Lemma~\ref{lem:mixed-root-relations}. The weight assertion follows directly from the opposite weights of $\xp{i}{r}$ and $\xm{i}{0}$.\end{proof}

\begin{lemma}\label{lem:diagonal-cartan-root}
Assume, in addition, the already established same-sign current relations at each even vertex and isotropy at every odd vertex. Then all diagonal Cartan--root relations hold.
\end{lemma}

\begin{proof} First let $i$ be even. Define the positive recurrence defect
\[
G(r,s)=[h_{i,r+1},\xp{i}{s}]-[h_{i,r},\xp{i}{s+1}]-\frac{\hbar b_{ii}}2\{h_{i,r},\xp{i}{s}\}.
\]
Taking the commutator of
\[
[\xp{i}{r+1},\xp{i}{s}]-[\xp{i}{r},\xp{i}{s+1}]-\frac{\hbar b_{ii}}2\{\xp{i}{r},\xp{i}{s}\}=0
\]
with $\xm{i}{0}$, and using the proved mixed relations, gives $G(r,s)+G(s,r)=0$. The $\HH{i}$ shift identity gives $G(0,s)=0$ for every $s$, because $[h_{i,0}^2,\xp{i}{s}]=b_{ii}\{h_{i,0},\xp{i}{s}\}$. Hence $G(r,0)=0$. Choose $k$ with $b_{ki}\ne0$. By Corollary~\ref{cor:cartan-shift-commutativity},
\[
[\HH{k},G(r,s)]=b_{ki}G(r,s+1),
\]
so induction on $s$ proves $G(r,s)=0$.

Now let $i$ be odd. Isotropy and the super Jacobi identity give
\[
[h_{i,r},\xp{i}{0}]=[[\xp{i}{r},\xm{i}{0}],\xp{i}{0}]=[\xp{i}{r},h_{i,0}]=0.
\]
Here $b_{ii}=0$. The same shift argument gives $[h_{i,r},\xp{i}{s}]=0$ for every $s$. The corresponding assertions for the negative generators follow either by the identical calculation or by the established anti-involution interchanging $\xp{i}{r}$ and $\xm{i}{r}$ and fixing $h_{i,r}$; the latter now fixes all $h_{i,r}$ by Lemma~\ref{lem:mixed-root-relations}.

The negative off-diagonal Cartan--root recurrence also holds at this point: apply the anti-involution of Lemma~\ref{lem:opposite-anti-involution} to Proposition~\ref{prop:off-diagonal-cartan-root}. Lemma~\ref{lem:mixed-root-relations} identifies $\tau(h_{i,r})$ with $h_{i,r}$ for every $r$.\end{proof}

\begin{lemma}\label{lem:cartan-commutativity}
One has $[h_{i,r},h_{j,s}]=0$ for all $i,j,r,s$.
\end{lemma}

\begin{proof} The Cartan--root recurrence is now proved for every pair $i,j$. For fixed $i,j$, apply Lemma~\ref{lem:logarithmic-shift-operators} with $b=b_{ij}$ to both signs. For every $n$, the polynomial $Q_{i,n}^{(j)}=h_{i,n}+P(h_{i,0},\ldots,h_{i,n-1})$ satisfies
\[
\begin{aligned}
[Q_{i,n}^{(j)},h_{j,s}]&=[Q_{i,n}^{(j)},[\xp{j}{s},\xm{j}{0}]]\\
&=b_{ij}\bigl([\xp{j}{s+n},\xm{j}{0}]-[\xp{j}{s},\xm{j}{n}]\bigr)=0.
\end{aligned}
\]
This remains valid when $b_{ij}=0$: the logarithmic shift lemma then says that each $Q_{i,n}^{(j)}$ commutes with both root families. Induction on $n$ removes the polynomial in lower $h_{i,r}$ and proves the assertion for every $s$.\end{proof}

\begin{theorem}\label{thm:presentation-equivalence}
For the signed cyclic root data under consideration, the natural homomorphism
\[
\pi:\Y^{\min}(\Pi)\longrightarrow\Y^D(\Pi)
\]
is an isomorphism. In particular, the reconstruction formulas \eqref{eq:current-generator-recursion} define in $\Y^{\min}(\Pi)$ all the higher generators of the full presentation.
\end{theorem}

\begin{proof} Propositions \ref{prop:off-diagonal-current-relations}, \ref{prop:diagonal-same-sign-relations}, \ref{prop:serre-propagation}, and \ref{prop:quartic-propagation}, together with Lemmas \ref{lem:cartan-shift-action}, \ref{lem:mixed-root-relations}, \ref{lem:diagonal-cartan-root}, and \ref{lem:cartan-commutativity}, prove all the full relations for the reconstructed elements of $\Y^{\min}(\Pi)$. For the negative off-diagonal Cartan--root relation, apply the anti-involution of Lemma~\ref{lem:opposite-anti-involution} to Proposition~\ref{prop:off-diagonal-cartan-root}; Lemma~\ref{lem:mixed-root-relations} now identifies $\tau(h_{i,r})$ with $h_{i,r}$ for every $r$. Thus the universal property of the full presentation gives a homomorphism $\rho:\Y^D(\Pi)\to\Y^{\min}(\Pi)$ taking each full current generator to its recursively constructed element. The composition $\rho\pi$ fixes the degree-zero-and-one generators of $\Y^{\min}(\Pi)$, so it is the identity. The composition $\pi\rho$ fixes the equally named generators of degrees zero and one of $\Y^D(\Pi)$; these generate $\Y^D(\Pi)$ by Proposition~\ref{prop:degree-one-generation}, so this composition is also the identity.\end{proof}
\par
Henceforth write $Y_\hbar(\Pi)$ for this common algebra. The degree-zero
generators together with all $\widetilde h_{i,1}$ generate the algebra,
since the shift relation recovers the degree-one root generators.

\section{Affine Lie superalgebras and degree-zero reflections}\label{sec:affine-lie-superalgebras-and-degree-zero-reflections}

\subsection{The derived and full affine Lie superalgebras}\label{subsec:the-derived-and-full-affine-lie-superalgebras}

For $m,n\geq2$ and $m\neq n$, normalize the even invariant supersymmetric
form on $\mathfrak{sl}(m|n)$ by
\[
  \kappa(a,b)=\operatorname{str}(ab).
\]
We use distinct notation for the two affine Lie superalgebras
\begin{equation}
  \mathfrak g'
  =\bigl(\mathfrak{sl}(m|n)\otimes\mathbb C[z,z^{-1}]\bigr)
    \oplus\mathbb Cc,
  \qquad
  \mathfrak g=\mathfrak g'\rtimes\mathbb Cd.
\label{eq:affine-lie-superalgebras}
\end{equation}
Their brackets are
\begin{equation}
\begin{aligned}
  [a\otimes z^r,b\otimes z^s]
    &=[a,b]\otimes z^{r+s}+r\delta_{r+s,0}\kappa(a,b)c,\\
  [d,a\otimes z^r]&=ra\otimes z^r,
  & [c,\mathfrak g]&=0.
\end{aligned}
\label{eq:affine-central-bracket}
\end{equation}
The Chevalley generators $e_i,f_i,h_i$, $i\in I$, generate
$\mathfrak g'$; they do not generate the degree derivation $d$.
The form extends to $\mathfrak g$ by
\[
\begin{gathered}
  (a\otimes z^r,b\otimes z^s)=\delta_{r+s,0}\kappa(a,b),
  \qquad(c,d)=1,\qquad(c,c)=(d,d)=0,\\
  \bigl(\mathfrak{sl}(m|n)\otimes\mathbb C[z,z^{-1}],
      \mathbb Cc+\mathbb Cd\bigr)=0.
\end{gathered}
\]
It is nondegenerate on $\mathfrak g$ and on its full Cartan subalgebra
$\mathfrak h=\mathfrak h_{\mathfrak{sl}(m|n)}\oplus\mathbb Cc\oplus\mathbb Cd$.

\begin{lemma}\label{lem:derived-cartan-degeneracy}
The restriction of the displayed invariant form to
$\mathfrak h'=\mathfrak h_{\mathfrak{sl}(m|n)}\oplus\mathbb Cc$ is degenerate.
More generally, every invariant form on $\mathfrak g'$ annihilates the
pairing of $c$ with every Chevalley Cartan generator.
\end{lemma}

\begin{proof} Since $h_i=[e_i,f_i]$ and $c$ is central, invariance gives
$(c,h_i)=(c,[e_i,f_i])=([c,e_i],f_i)=0$.
For the displayed affine form, $c$ is also orthogonal to every loop element.
\end{proof}

\par\medskip
Thus dual Cartan bases in a Casimir formula are taken in $\mathfrak h$,
not in $\mathfrak h'$. Introducing $d$ for this purpose does not assert
that $d$ belongs to the Yangian under consideration.

\subsection{Ordered bases and normalized odd reflections}\label{subsec:ordered-bases-and-normalized-odd-reflections}

The ordered basis and its periodic continuation below make the normalization uniform at every cyclic
vertex, including vertex $0$. All brackets use the central cocycle in
\eqref{eq:affine-central-bracket}.

\begin{definition}\label{def:periodic-loop-basis}
Let $v_1,\ldots,v_N$ be a homogeneous basis of $\mathbb C^{m|n}$.
Choose an ordered basis $u_1,\ldots,u_N$ of
$\mathbb C[z,z^{-1}]\otimes\mathbb C^{m|n}$ over
$\mathbb C[z,z^{-1}]$, of the form
\[
  u_a=z^{k_a}v_{\rho(a)}\qquad(1\leq a\leq N),
\]
where $k_a\in\mathbb Z$ and $\rho(1),\ldots,\rho(N)$ is a permutation of $1,\ldots,N$.
Extend it to a sequence indexed by $\mathbb Z$ by
$u_{a+N}=z^{-1}u_a$; thus $\rho(a+N)=\rho(a)$ and
$k_{a+N}=k_a-1$. Put $p_a=p(v_{\rho(a)})$. Its lifted weights and
matrix units are
\[
  w_a=\varepsilon_{\rho(a)}+k_a\delta,\qquad
  \mathcal E_{ab}=z^{k_a-k_b}E_{\rho(a),\rho(b)}.
\]
Thus $w_{a+N}=w_a-\delta$ and $\mathcal E_{a+N,b+N}=\mathcal E_{ab}$.
We use $(\varepsilon_a,\varepsilon_b)=(-1)^{p(v_a)}\delta_{ab}$ and
$(\delta,\varepsilon_a)=(\delta,\delta)=0$.

\par\medskip
For $i\in I$, choose any integer lift and set
\begin{equation}
\alpha_i=w_i-w_{i+1},\quad
e_i=\mathcal E_{i,i+1},\quad
f_i=(-1)^{p_i}\mathcal E_{i+1,i},\quad
h_i=[e_i,f_i].\label{eq:periodic-chevalley-generators}
\end{equation}
These definitions do not depend on the lift. More explicitly,
\begin{equation}
h_i=(-1)^{p_i}E_{\rho(i),\rho(i)}
-(-1)^{p_{i+1}}E_{\rho(i+1),\rho(i+1)}+(k_i-k_{i+1})c.
\label{eq:periodic-cartan-generators}
\end{equation}
Consequently $\sum_{i\in I}h_i=c$. For $B=(b_{ij})$,
$b_{ij}=(\alpha_i,\alpha_j)$, direct calculation gives
\[
[e_i,f_j]=\delta_{ij}h_i,\qquad
[h_i,e_j]=b_{ij}e_j,\qquad [h_i,f_j]=-b_{ij}f_j.
\]
In particular $b_{ii}=(-1)^{p_i}+(-1)^{p_{i+1}}$, and $\alpha_i$ is odd
precisely when $(-1)^{p_{i+1}}=-(-1)^{p_i}$.
\end{definition}

For the initial sequence $u_a=z^{-q(a)}v_{\bar a}$, where
$q(a)=\lfloor(a-1)/N\rfloor$ and $1\leq\bar a\leq N$, the vertex $0$
is represented by $i=N$, and
\[
e_0=zE_{N1},\qquad
f_0=(-1)^{p_N}z^{-1}E_{1N},\qquad
h_0=(-1)^{p_N}E_{NN}-(-1)^{p_1}E_{11}+c.
\]
Allowing these general basis choices is essential: a reflection at the affine
vertex changes loop powers as well as the order of finite weights.

\begin{proposition}\label{prop:classical-odd-reflection}
Suppose $\alpha_i$ is odd. Let $\pi_i:\mathbb Z\to\mathbb Z$
exchange $i+kN$ and $i+1+kN$ for every $k\in\mathbb Z$ and fix all
other integers. Use the reflected sequence $u'_a=u_{\pi_i(a)}$ and
denote its generators by $e'_j,f'_j,h'_j$. Its simple roots are
\[
\alpha'_i=-\alpha_i,\qquad
\alpha'_{i-1}=\alpha_{i-1}+\alpha_i,\qquad
\alpha'_{i+1}=\alpha_i+\alpha_{i+1},
\]
with the other roots unchanged. The identity on the common loop
superalgebra gives an isomorphism $R_i:\mathfrak g'_{\Pi}\to\mathfrak g'_{\Pi'}$
between its two root-labelled copies, expressed as follows:
\begin{align}
R_i(e_i)&=-(-1)^{p_i}f'_i,&
R_i(f_i)&=(-1)^{p_i}e'_i,&
R_i(h_i)&=-h'_i,\label{eq:classical-reflection-at-vertex}\\
R_i(e_{i-1})&=[e'_{i-1},e'_i],&
R_i(f_{i-1})&=-(-1)^{p_i}[f'_i,f'_{i-1}],&
R_i(h_{i-1})&=h'_{i-1}+h'_i,\label{eq:classical-reflection-left-neighbor}\\
R_i(e_{i+1})&=[e'_i,e'_{i+1}],&
R_i(f_{i+1})&=(-1)^{p_i}[f'_{i+1},f'_i],&
R_i(h_{i+1})&=h'_i+h'_{i+1}.\label{eq:classical-reflection-right-neighbor}
\end{align}
All remaining generators are sent to the corresponding primed
generators, and $R_i(c)=c$. The reflection formed from the primed
sequence is the inverse of $R_i$.
\end{proposition}

\begin{proof} The formula for the simple roots follows by exchanging the
two lifted weights. At the reflected vertex,
\[
e'_i=\mathcal E_{i+1,i},\qquad
f'_i=(-1)^{p_{i+1}}\mathcal E_{i,i+1}=-(-1)^{p_i}e_i,
\]
which proves the first two identities in \eqref{eq:classical-reflection-at-vertex}.
Since $e_i,f_i$ are odd, $[f_i,e_i]=[e_i,f_i]$, and hence $h'_i=-h_i$.

For the left neighbour,
\[
e'_{i-1}=\mathcal E_{i-1,i+1},\qquad e'_i=\mathcal E_{i+1,i},
\]
so $[e'_{i-1},e'_i]=\mathcal E_{i-1,i}=e_{i-1}$. Similarly,
\[
 [f'_i,f'_{i-1}]=(-1)^{p_{i+1}}(-1)^{p_{i-1}}\mathcal E_{i,i-1}
             =-(-1)^{p_i}f_{i-1}.
\]
For the right neighbour, the same multiplication gives
\[
 [e'_i,e'_{i+1}]=e_{i+1},\qquad
 [f'_{i+1},f'_i]=(-1)^{p_i}f_{i+1}.
\]
There are no central terms in these four brackets, since the matrix products have zero
supertrace. The Cartan identities, including their central coefficients, follow by telescoping
\eqref{eq:periodic-cartan-generators}. All calculations remain valid across the index $0$ because the sequence is periodic. They are
identities in the common centrally extended loop superalgebra, so they give a Lie superalgebra
isomorphism. Exchanging the same pair again restores the original sequence; this also proves the
inverse assertion.\end{proof}

\section{Quantum reflection isomorphisms}\label{sec:quantum-reflection-isomorphisms}
\subsection{Even braid automorphisms}\label{subsec:even-braid-automorphisms}
The full current relations and the generation result established above suffice for the even braid
construction.

\begin{theorem}\label{thm:even-braid-automorphisms}
For an even simple root $\alpha_i$, the generators
\[
 x^+_{i,0},\qquad (-1)^{p_i}x^-_{i,0},\qquad (-1)^{p_i}h_{i,0}
\]
satisfy the $\mathfrak{sl}_2$ relations. The adjoint actions of $x^+_{i,0}$ and
$(-1)^{p_i}x^-_{i,0}$ are locally nilpotent on $\Y(\Pi)$, and
\begin{equation}
 T_i=\exp(\operatorname{ad}x^+_{i,0})
 \exp\bigl(-\operatorname{ad}((-1)^{p_i}x^-_{i,0})\bigr)
 \exp(\operatorname{ad}x^+_{i,0})\label{eq:even-braid-lift}
\end{equation}
is an algebra automorphism. For even vertices $i,j$,
\[
 T_iT_j=T_jT_i\quad\text{if }b_{ij}=0,\qquad
 T_iT_jT_i=T_jT_iT_j\quad\text{if }i,j\text{ are adjacent}.
\]
Thus these maps give a braid group action for the subsystem generated by the even simple roots
of the fixed base $\Pi$.

If $a$ has integral $(-1)^{p_i}h_{i,0}$-weight $k$, then
\begin{equation}
  T_i^2(a)=(-1)^ka.\label{eq:even-braid-square}
\end{equation}
In particular $T_i^2(x^\pm_{j,r})=-x^\pm_{j,r}$ for a neighbour $j$ of $i$. Whenever such a generator is nonzero,
the corresponding action does not factor through the Weyl group.
\end{theorem}

\begin{proof}
Since $\alpha_i$ is even, $b_{ii}=2(-1)^{p_i}$, giving
\[
 [x^+_{i,0},(-1)^{p_i}x^-_{i,0}]=(-1)^{p_i}h_{i,0},\qquad
 [(-1)^{p_i}h_{i,0},x^+_{i,0}]=2x^+_{i,0},
\]
\[
 [(-1)^{p_i}h_{i,0},(-1)^{p_i}x^-_{i,0}]=-2(-1)^{p_i}x^-_{i,0}.
\]
We verify local nilpotence on the generating set in Proposition~\ref{prop:degree-one-generation}. It holds on the degree-zero
generators by the degree-zero Serre relations. For $j\ne i$, the Serre relations in all current
degrees imply $(\operatorname{ad}x^+_{i,0})^2(x^+_{j,1})=0$ when $j$ is adjacent to $i$, and $[x^+_{i,0},x^+_{j,1}]=0$ otherwise. The
mixed-sign relation gives $[x^+_{i,0},x^-_{j,1}]=0$ for $j\ne i$. For the remaining current generators, the
difference relations give
\[
\begin{aligned}
 [x^+_{i,0},x^+_{i,1}]&=-\hbar(-1)^{p_i}(x^+_{i,0})^2,\\
 [x^+_{i,0},h_{j,1}]&=-b_{ji}x^+_{i,1}-\frac{\hbar b_{ji}}2\{h_{j,0},x^+_{i,0}\},\\
 [x^+_{i,0},x^-_{i,1}]&=h_{i,1}.
\end{aligned}
\]
Here $x^+_{i,0},h_{j,0}$ are even, so the braces in the second line are ordinary anticommutators. Since
$[x^+_{i,0},h_{j,0}]=-b_{ji}x^+_{i,0}$, these identities show $(\operatorname{ad}x^+_{i,0})^3(h_{j,1})=0$ and $(\operatorname{ad}x^+_{i,0})^4(x^-_{i,1})=0$. The
negative-sign relations prove the corresponding assertions for $\operatorname{ad}((-1)^{p_i}x^-_{i,0})$. A locally nilpotent even
derivation on a generating set is locally nilpotent on every polynomial in that set, by the
Leibniz rule. The exponentials in \eqref{eq:even-braid-lift} are therefore well-defined, with inverse exponentials of
opposite sign.

The adjoint action is a weight representation: the weight of $x^\pm_{j,r}$ under $(-1)^{p_i}h_{i,0}$ is $\pm(-1)^{p_i}b_{ij}$,
independently of $r$. Together with the local nilpotence just proved, this gives an integrable
representation of each even rank-one and rank-two subalgebra. For nonadjacent vertices
these rank-one actions commute. For adjacent even vertices their normalized generators form
$\mathfrak{sl}_3$. Each vector lies in a finite-dimensional $\mathfrak{sl}_3$-submodule, which integrates to $SL_3(\mathbb C)$. The
elementary representatives
$\exp(x^+_{i,0})\exp(-(-1)^{p_i}x^-_{i,0})\exp(x^+_{i,0})$
satisfy the braid relation in $SL_3$; hence
their operators do so on every such submodule. This proves the asserted braid relations.

In $SL_2$, the square of the same representative is $-I$. Its action on an integral weight-$k$
vector is $(-1)^k$, proving \eqref{eq:even-braid-square}. At an adjacent vertex $(-1)^{p_i}b_{ij}=-1$, which gives the last
assertion.\end{proof}

\subsection{Normalized odd-reflection formulas}\label{subsec:normalized-odd-reflection-formulas}
We now lift Proposition~\ref{prop:classical-odd-reflection} to the full current algebra. For finite-type super Yangians, Gow established the RTT--Drinfeld correspondence via Gauss decomposition for the distinguished parity sequence \cite[Theorem~3]{Gow}. For arbitrary parity sequences, our proof uses Tsymbaliuk's realization \cite[Theorem~2.67]{Tsymbaliuk} and its polynomial parameter form \cite[Remark~2.82]{Tsymbaliuk}, after identifying the normalization explicitly. These RTT statements concern the non-affine super Yangian; the passage to the cyclic affine diagram is supplied by Lemma~\ref{lem:cut-diagram-locality}.

Let $i$ be odd, and let $\Pi'=s_i\Pi$, where $s_i$ exchanges the adjacent positions
of the parity sequence. Primes denote target generators, and throughout this odd-reflection
calculation $b_{ab}$ denotes the source matrix entry. In particular,
$b_{i-1,i}=-(-1)^{p_i}$ and $b_{i+1,i}=(-1)^{p_i}$.
The superanticommutator of the two odd target root generators is
\[
 \{x^{\prime+}_{i,0},x^{\prime-}_{i,0}\}
 =x^{\prime+}_{i,0}x^{\prime-}_{i,0}-x^{\prime-}_{i,0}x^{\prime+}_{i,0}.
\]
\begin{theorem}\label{thm:quantum-odd-reflections}
For every odd vertex $i$, the following
assignments extend to a $\mathbb C[\hbar]$-superalgebra isomorphism
\[
 T_i:\Y(\Pi)\longrightarrow\Y(\Pi').
\]
On degree-zero root generators they are
\begin{align}
 T_i(x^+_{i,0})&=-(-1)^{p_i}x^{\prime-}_{i,0},&
 T_i(x^-_{i,0})&=(-1)^{p_i}x^{\prime+}_{i,0},\label{eq:odd-reflection-degree-zero-vertex}\\
 T_i(x^+_{i-1,0})&=[x^{\prime+}_{i-1,0},x^{\prime+}_{i,0}],&
 T_i(x^-_{i-1,0})&=-(-1)^{p_i}[x^{\prime-}_{i,0},x^{\prime-}_{i-1,0}],\label{eq:odd-reflection-degree-zero-left}\\
 T_i(x^+_{i+1,0})&=[x^{\prime+}_{i,0},x^{\prime+}_{i+1,0}],&
 T_i(x^-_{i+1,0})&=(-1)^{p_i}[x^{\prime-}_{i+1,0},x^{\prime-}_{i,0}].\label{eq:odd-reflection-degree-zero-right}
\end{align}
All other degree-zero root generators are fixed. The Cartan assignments are
\begin{equation}
 T_i(h_{i,0})=-h'_{i,0},\qquad
 T_i(\widetilde h_{i,1})=-\widetilde h'_{i,1},\label{eq:odd-reflection-cartan-vertex}
\end{equation}
\begin{equation}
 \begin{aligned}
 T_i(h_{j,0})&=h'_{j,0}+h'_{i,0},\\
 T_i(\widetilde h_{j,1})&=\widetilde h'_{j,1}+\widetilde h'_{i,1}
          +\frac{\hbar b_{ji}}2\{x^{\prime+}_{i,0},x^{\prime-}_{i,0}\}
          \quad(j=i-1,i+1),
 \end{aligned}\label{eq:odd-reflection-cartan-neighbors}
\end{equation}
and the remaining $h_{j,0}$ and $\widetilde h_{j,1}$ are fixed. The corresponding degree-one root assignments are
\begin{align}
 T_i(x^+_{i,1})&=-(-1)^{p_i}
       (x^{\prime-}_{i,1}-\hbar h'_{i,0}x^{\prime-}_{i,0}),\notag\\
 T_i(x^-_{i,1})&=(-1)^{p_i}
       (x^{\prime+}_{i,1}-\hbar h'_{i,0}x^{\prime+}_{i,0}),\label{eq:odd-reflection-degree-one-vertex}\\
 T_i(x^+_{i-1,1})&=[x^{\prime+}_{i-1,1},x^{\prime+}_{i,0}],\notag\\
 T_i(x^-_{i-1,1})&=-(-1)^{p_i}[x^{\prime-}_{i,0},x^{\prime-}_{i-1,1}],\label{eq:odd-reflection-degree-one-left}\\
 T_i(x^+_{i+1,1})&=[x^{\prime+}_{i,0},x^{\prime+}_{i+1,1}],\notag\\
 T_i(x^-_{i+1,1})&=(-1)^{p_i}[x^{\prime-}_{i+1,1},x^{\prime-}_{i,0}].\label{eq:odd-reflection-degree-one-right}
\end{align}
All other degree-one root generators are fixed, and $T_i(h_{j,1})$ is recovered from
\[
 T_i(h_{j,1})=T_i(\widetilde h_{j,1})+\frac\hbar2T_i(h_{j,0})^2.
\]
The same table for the reverse reflection is the inverse. The maps and their inverses preserve
the current filtration $\deg x^\pm_{a,r}=\deg h_{a,r}=r$, $\deg\hbar=0$. The statement includes the
affine vertex.
\end{theorem}

\subsection{The finite RTT normalization}\label{subsec:the-finite-rtt-normalization}
We prove the required finite calculation before applying locality. For this calculation, number
the weights of a finite ordered basis by $1,\ldots,N$. Let $T(v)$ be the RTT matrix and write its
ordinary matrix Gauss decomposition as $T(v)=F(v)D(v)E(v)$. The entry convention is the
one in which matrix multiplication is ordinary; the coefficient algebra and tensor products
remain graded. Write
\[
\begin{aligned}
 t_{ab}(v)&=\delta_{ab}+\hbar t_{ab,0}v^{-1}+\hbar t_{ab,1}v^{-2}+O(v^{-3}),\\
 D_a(v)&=1+\hbar A_av^{-1}+\hbar B_av^{-2}+O(v^{-3}),\\
 Z_a&=B_a-\frac\hbar2A_a^2,\qquad
 \gamma_j=-\frac\hbar2\sum_{a=1}^j(-1)^{p_a}.
\end{aligned}
\]
Thus $A_a=t_{aa,0}$, and the order-two Gauss identity is
\begin{equation}
 B_a=t_{aa,1}-\hbar\sum_{b<a}t_{ab,0}t_{ba,0}.\label{eq:gauss-second-coefficient}
\end{equation}
Let $E^{\mathrm G}_{j,0}=t_{j,j+1,0}$ and $F^{\mathrm G}_{j,0}=t_{j+1,j,0}$ denote the initial coefficients of the simple upper and
lower Gauss currents. The superscript $\mathrm G$ distinguishes them from matrix units.

\begin{lemma}\label{lem:finite-rtt-normalization}
The finite Drinfeld algebra in our normalization has
an RTT embedding for which
\begin{equation}
 \begin{aligned}
 x^+_{j,0}&\longmapsto(-1)^{p_j}E^{\mathrm G}_{j,0},\qquad
 x^-_{j,0}\longmapsto(-1)^{p_j+p_{j+1}}F^{\mathrm G}_{j,0},\\
 h_{j,0}&\longmapsto A_j-A_{j+1},
 \end{aligned}\label{eq:finite-rtt-degree-zero-embedding}
\end{equation}
\begin{equation}
 \widetilde h_{j,1}\longmapsto Z_{j+1}-Z_j+\gamma_j(A_j-A_{j+1}).\label{eq:finite-rtt-shifted-cartan-embedding}
\end{equation}
Its image is the RTT $\mathfrak{sl}$ subalgebra.
\end{lemma}

\begin{proof} We use the finite Drinfeld--RTT isomorphism \cite[Theorem~2.67]{Tsymbaliuk}, with the polynomial
parameter supplied by its Rees form \cite[Remark~2.82]{Tsymbaliuk}. Here is the normalization conversion,
which is needed for the subsequent signs. If $X^\pm_{j,r},\mathsf H_{j,r}$ are the generators in that isomorphism,
compose with
\begin{equation}
 \begin{aligned}
 x^+_{j,r}&\longmapsto(-1)^rX^-_{j,r},\qquad
 x^-_{j,r}\longmapsto(-1)^r(-1)^{p_j+p_{j+1}}X^+_{j,r},\\
 h_{j,r}&\longmapsto(-1)^{r+1}\mathsf H_{j,r}.
 \end{aligned}\label{eq:finite-drinfeld-normalization-conversion}
\end{equation}
This change preserves the defining relations: exchanging the positive and negative currents
changes the sign in the current difference relation, and the factor $(-1)^r$ changes it back. For
the mixed bracket, super skew symmetry supplies the factor $-(-1)^{p_j+p_{j+1}}$; the displayed factor
in the image of $x^-_{j,r}$ gives the stated image of $h_{j,r+s}$. Weights, isotropic relations, and both
Serre families are preserved term by term. The change is invertible, since its square multiplies
both root currents at $j$ by $(-1)^{p_j+p_{j+1}}$ and fixes the Cartan currents.

In the RTT embedding, the original Cartan generating series is $D_j(v)^{-1}D_{j+1}(v)$ at $v=
u+\gamma_j$. The convention in \cite[(2.66)]{Tsymbaliuk} fixes the first shift to zero. Our shifts differ from
that convention by the single scalar $-\hbar(-1)^{p_1}/2$, hence are obtained by a common spectral
translation. Expanding the quotient gives
\[
 \mathsf H_{j,0}=A_{j+1}-A_j,\qquad
 \mathsf H_{j,1}=B_{j+1}-B_j+\hbar A_j^2-\hbar A_jA_{j+1}
                 -\gamma_j(A_{j+1}-A_j).
\]
The elements $A_a$ commute. Under \eqref{eq:finite-drinfeld-normalization-conversion}, $h_{j,0}$ maps to $-\mathsf H_{j,0}$ and $h_{j,1}$ to $\mathsf H_{j,1}$. Subtracting
$\hbar h_{j,0}^2/2$ proves \eqref{eq:finite-rtt-shifted-cartan-embedding}. The two initial root-current formulas follow from the upper and lower
Gauss-current formulas in the same embedding. Finally, all the stated operations preserve
the RTT $\mathfrak{sl}$ subalgebra: the original isomorphism has this image, and a common spectral
translation preserves invariance under scalar multiplication of the RTT matrix.\end{proof}

\begin{lemma}\label{lem:rtt-adjacent-swap}
For a finite ordered weight sequence, swapping
two consecutive weights of opposite parity induces exactly the degree-zero-and-one table in
Theorem~\ref{thm:quantum-odd-reflections}, with nonexistent endpoint neighbors omitted.
\end{lemma}

\begin{proof} Let $\sigma=(i,i+1)$. The RTT permutation isomorphism sends $t_{ab,r}$ to $t'_{\sigma(a),\sigma(b),r}$; it
preserves the RTT $\mathfrak{sl}$ subalgebra \cite[Lemma~2.24 and Corollary~2.38]{Tsymbaliuk}. Denote it by $\Phi$. The
degree-zero formulas are the signed matrix-unit identities of Proposition~\ref{prop:classical-odd-reflection}, using \eqref{eq:finite-rtt-degree-zero-embedding}.
Formula \eqref{eq:gauss-second-coefficient} gives
\[
\begin{aligned}
 \Phi(A_i)&=A'_{i+1},&\Phi(A_{i+1})&=A'_i,\\
 \Phi(Z_i)&=Z'_{i+1}+\hbar F^{\mathrm G\prime}_{i,0}E^{\mathrm G\prime}_{i,0},&
 \Phi(Z_{i+1})&=Z'_i-\hbar E^{\mathrm G\prime}_{i,0}F^{\mathrm G\prime}_{i,0}.
\end{aligned}
\]
Every other $A_a,Z_a$ is fixed: below the transposed pair neither summation in \eqref{eq:gauss-second-coefficient} changes, and
above it the two summands are merely exchanged. The primed root is odd and $(-1)^{p'_i}=-(-1)^{p_i}$,
so \eqref{eq:finite-rtt-degree-zero-embedding} gives $E^{\mathrm G\prime}_{i,0}=-(-1)^{p_i}x^{\prime+}_{i,0}$ and $F^{\mathrm G\prime}_{i,0}=-x^{\prime-}_{i,0}$. Consequently
\begin{equation}
 \begin{aligned}
 \Phi(Z_i)&=Z'_{i+1}+(-1)^{p_i}\hbar x^{\prime-}_{i,0}x^{\prime+}_{i,0},\\
 \Phi(Z_{i+1})&=Z'_i-(-1)^{p_i}\hbar x^{\prime+}_{i,0}x^{\prime-}_{i,0}.
 \end{aligned}\label{eq:rtt-transposition-gauss-coefficients}
\end{equation}
The inclusive prefix in the definition of $\gamma_j$ gives
\begin{equation}
 \gamma'_i-\gamma_i=(-1)^{p_i}\hbar,\qquad
 \gamma'_j=\gamma_j\quad(j\ne i).\label{eq:rtt-spectral-shift-change}
\end{equation}
This identity also holds when $i=1$.

Substitute \eqref{eq:rtt-transposition-gauss-coefficients} and \eqref{eq:rtt-spectral-shift-change} into \eqref{eq:finite-rtt-shifted-cartan-embedding}. At the reflected vertex the additional term is
\[
 (\gamma'_i-\gamma_i)h'_{i,0}
 -(-1)^{p_i}\hbar(x^{\prime+}_{i,0}x^{\prime-}_{i,0}+x^{\prime-}_{i,0}x^{\prime+}_{i,0})=0,
\]
so $\Phi(\widetilde h_{i,1})=-\widetilde h'_{i,1}$. At the left neighbor, the remaining term is
\[
 (-1)^{p_i}\hbar(x^{\prime-}_{i,0}x^{\prime+}_{i,0}-h'_{i,0}/2)
 =-\frac{(-1)^{p_i}\hbar}{2}\{x^{\prime+}_{i,0},x^{\prime-}_{i,0}\},
\]
whereas at the right neighbor it is
\[
 (-1)^{p_i}\hbar(x^{\prime+}_{i,0}x^{\prime-}_{i,0}-h'_{i,0}/2)
 =\frac{(-1)^{p_i}\hbar}{2}\{x^{\prime+}_{i,0},x^{\prime-}_{i,0}\}.
\]
Here $h'_{i,0}=x^{\prime+}_{i,0}x^{\prime-}_{i,0}+x^{\prime-}_{i,0}x^{\prime+}_{i,0}$. These are precisely \eqref{eq:odd-reflection-cartan-neighbors}, since $b_{i-1,i}=-(-1)^{p_i}$ and $b_{i+1,i}=(-1)^{p_i}$. All other $\widetilde h_{j,1}$ are
fixed.

It remains to identify the degree-one root generators. They are determined by
$[\widetilde h_{a,1},x^+_{j,0}]=b_{aj}x^+_{j,1}$ and
$[\widetilde h_{a,1},x^-_{j,0}]=-b_{aj}x^-_{j,1}$, using an $a$ for which $b_{aj}\ne0$. For $j=i\pm1$ take $a=i$. The
target $\widetilde h'_{i,1}$ commutes with $x^{\prime+}_{i,0},x^{\prime-}_{i,0}$, while $b'_{ij}=-b_{ij}$. Applying the super Jacobi identity to the
degree-zero bracket gives exactly \eqref{eq:odd-reflection-degree-one-left} and \eqref{eq:odd-reflection-degree-one-right}. For $j=i$, take either existing neighbor $a$.
The elementary identities
\[
\begin{aligned}
 [\{x^{\prime+}_{i,0},x^{\prime-}_{i,0}\},x^{\prime+}_{i,0}]&=2h'_{i,0}x^{\prime+}_{i,0},\\
 [\{x^{\prime+}_{i,0},x^{\prime-}_{i,0}\},x^{\prime-}_{i,0}]&=-2h'_{i,0}x^{\prime-}_{i,0}
\end{aligned}
\]
follow from $(x^{\prime+}_{i,0})^2=(x^{\prime-}_{i,0})^2=0$ and
$h'_{i,0}=x^{\prime+}_{i,0}x^{\prime-}_{i,0}+x^{\prime-}_{i,0}x^{\prime+}_{i,0}$; together with the neighbor Cartan formula they
give \eqref{eq:odd-reflection-degree-one-vertex}. If $j\notin\{i-1,i,i+1\}$, the additional terms $\widetilde h'_{i,1}$ and
$\{x^{\prime+}_{i,0},x^{\prime-}_{i,0}\}$ in any affected Cartan
image commute with $x^{\prime+}_{j,0},x^{\prime-}_{j,0}$. The unchanged target shift relation therefore shows that both
degree-one root generators at $j$ are fixed. This proves the full table, including endpoint cases.
Since all calculations arise from an RTT isomorphism, they preserve every finite defining
relation.\end{proof}

\subsection{Finite locality of the defining-relation checks}\label{subsec:finite-locality-of-the-defining-relation-checks}
The following argument passes from the finite-type calculation of a reflection table to its use in
the affine algebra.

For $k\in I$, put $J_k=I\setminus\{k\}$ and cut the cyclic diagram at $k$. Its remaining vertices form a
finite path. Let $\Y^{\mathrm{fin}}(\Pi;k)$ be the finite-type Drinfeld super Yangian for this ordered path, with
the restricted root inner products $b_{ab}$, $a,b\in J_k$, and the inherited parities. The underlying
ordered weight sequence still contains $m$ even and $n$ odd weights. The finite-type generators
will be denoted $h^{(k)}_{a,r},x^{(k),\pm}_{a,r}$.

\begin{lemma}\label{lem:finite-path-homomorphisms}
For every $k$, the assignments
\[
 h^{(k)}_{a,r}\longmapsto h_{a,r},\qquad
 x^{(k),\pm}_{a,r}\longmapsto x^\pm_{a,r}
 \quad(a\in J_k,r\ge0)
\]
define a homomorphism
\[
 \iota_{\Pi,k}:\Y^{\mathrm{fin}}(\Pi;k)\longrightarrow\Y(\Pi).
\]
\end{lemma}

\begin{proof} We use the finite-type current presentation with the super bracket and super anticommutator
conventions fixed above. Its Cartan, weight, mixed, and current difference relations
are the corresponding restrictions of \eqref{eq:full-cartan-commutators}--\eqref{eq:full-root-current-shift}. Its same-sign brackets at a zero Cartan entry
are among \eqref{eq:full-isotropic-relations} and the $q_{ab}=1$ cases of \eqref{eq:full-serre-relations}. Its cubic Serre relations and quartic relations
on triples contained in $J_k$ are likewise included in \eqref{eq:full-serre-relations} and \eqref{eq:full-quartic-serre-relations}.

One frequently used finite-type presentation additionally states $[h_{a,r},x^\pm_{a,s}]=0$ explicitly
at every isotropic vertex, rather than retaining only its difference form. This identity also
holds in $\Y(\Pi)$: it was proved in the reconstruction in Lemma~\ref{lem:diagonal-cartan-root}, and is transported by
Theorem~\ref{thm:presentation-equivalence}. Thus this alternative convention gives the same homomorphism. All required
finite-type relations therefore vanish after the assignments, and the universal property of the
presentation proves the claim. The proof is independent of whether $0\in J_k$: after cutting, the
restricted matrix and parity data are those of the same finite ordered path.\end{proof}

\begin{lemma}\label{lem:cut-diagram-locality}
Let $i$ be an odd vertex and let $\Pi'=s_i\Pi$ be
its reflected cyclic root data. Suppose a table $F_i$ assigns to every degree-zero-and-one generator
at a vertex $a$ of $\Y^{\min}(\Pi)$ a polynomial in the equally graded generators of $\Y(\Pi')$. Assume the
following two properties.

\begin{enumerate}
 \item Each image involves only vertices in $\{a,i\}$. In particular, the image of a generator at $i$
 involves only vertex $i$.
 \item For every cut $k\ne i$, there is a homomorphism
 \[
 \Phi_i^{(k)}:\Y^{\mathrm{fin}}(\Pi;k)\longrightarrow\Y^{\mathrm{fin}}(\Pi';k)
 \]
 whose values on the generators of degrees zero and one are exactly the same polynomial
 table $F_i$, restricted to $J_k$.
\end{enumerate}
Then $F_i$ extends uniquely to a homomorphism
\[
 T_i:\Y(\Pi)\longrightarrow\Y(\Pi').
\]
If a reverse table has the same two properties and the corresponding finite maps are inverse to
$\Phi_i^{(k)}$ for every $k\ne i$, then the resulting affine homomorphisms are mutually inverse.
\end{lemma}

\begin{proof} Consider a defining relation $R$ of the degree-zero-and-one presentation. Its set $S$ of
vertex indices has at most two members, except for an extra Serre relation, for which $S$ is the
triple of its consecutive vertices. Hence
\[
 |S\cup\{i\}|\le4<N.
\]
Choose $k\notin S\cup\{i\}$. This works also when $N=5$. The source relation $R$ is a valid identity in
$\Y^{\mathrm{fin}}(\Pi;k)$: the degree-zero-and-one relations are consequences of its full current presentation.
If $R$ is a quartic relation, its entire consecutive triple is retained in the cut path. When one
or both neighboring roots are odd, the requisite finite identity is explicitly covered by \cite[Remark~2.61(a)]{Tsymbaliuk}, not just the displayed relation (2.60) there for even neighbors.

Applying $\Phi_i^{(k)}$ to $R$ gives the zero identity in the target finite algebra. By the support
hypothesis, every polynomial in this identity uses only the retained vertices. Applying $\iota_{\Pi',k}$
therefore gives precisely $F_i(R)=0$ in $\Y(\Pi')$. Thus the table annihilates every defining relation
of $\Y^{\min}(\Pi)$ and gives a homomorphism $\Y^{\min}(\Pi)\to\Y(\Pi')$. The presentation equivalence
transports it to the asserted $T_i$. Uniqueness follows from generation in degrees zero and one.

For the inverse assertion, apply the same construction to the reverse table. Fix a degree-zero-and-one
generator at a vertex $a$, and choose a cut outside $\{a,i\}$. Both substitutions and
their composite involve only $\{a,i\}$. The finite inverse identity for this generator thus transports
through the cut map to the corresponding affine inverse identity. The two composites fix all
degree-zero-and-one generators, which generate the two affine algebras. Both composites are
therefore the identity.\end{proof}

\subsection{Passage to the cyclic diagram and inverse}\label{subsec:passage-to-the-cyclic-diagram-and-inverse}
\begin{proof}[Proof of Theorem~\ref{thm:quantum-odd-reflections}] For each cut $k\ne i$, order the remaining path and use Lemma~\ref{lem:rtt-adjacent-swap}. Its
formula table is independent of the position of the cut: all shifts are given by the inclusive prefix, and the final answers contain only $(-1)^{p_i}$, $b_{ji}$, and the displayed target generators. In particular, no residual common spectral shift remains when the reflected root is at an endpoint of a cut path. Every image at a vertex $a$ is supported on $\{a,i\}$. Lemma~\ref{lem:cut-diagram-locality} therefore proves that the table preserves every affine minimal relation and extends to a homomorphism of the full current algebras. This covers $i=0$ as well: cut at a different vertex and apply the same finite ordered-path calculation.

For completeness, invertibility also follows immediately from the displayed generators. The reverse reflection has sign $-(-1)^{p_i}$, and its degree-zero table is inverse to the forward one. On the reflected odd pair the reverse substitution sends
\[
\{x'^{+}_{i,0},x'^{-}_{i,0}\}
\quad\text{to}\quad
\{x^{+}_{i,0},x^{-}_{i,0}\};
\]
indeed, the two individual products are sent to $-x^{-}_{i,0}x^{+}_{i,0}$ and $-x^{+}_{i,0}x^{-}_{i,0}$. The reflected Cartan has image $-\widetilde h_{i,1}$ and the reverse neighbor coefficient is $b'_{ji}=-b_{ji}$. Thus the two quadratic corrections in every neighboring $\widetilde h_{j,1}$ cancel. Both composites fix all degree-zero generators and all $\widetilde h_{j,1}$. These generate the Yangian by the shift relation, so both composites are the identity.

Finally, the displayed maps preserve the current filtration on the degree-zero-and-one generators. Each quadratic correction in a degree-one image has current degree zero. For higher current degrees use the recursion \eqref{eq:current-generator-recursion}: if the image of $x^{\pm}_{a,r}$ has degree at most $r$, its bracket with the image of $\widetilde h_{\nu(a),1}$, which has degree at most one, has degree at most $r+1$. The mixed-bracket formula gives the same bound for the Cartan generators. Induction proves filteredness in every current degree. The identical argument applies to the inverse.

\end{proof}

\section{The extended reflection groupoid}\label{sec:the-extended-reflection-groupoid}

\begin{definition*}\phantomsection\label{def:extended-reflection-groupoid}
Let $\mathcal P$ be the set of parity sequences with $m$ even and $n$ odd entries, and put
\[
\mathcal D=\{1,-1\}^{N}/\langle(-1,\ldots,-1)\rangle.
\]
Let $W_{\mathrm{aff}}$ be the affine Coxeter group of type $\widetilde A_{N-1}$. Its simple generators $s_i$ swap cyclic positions $i,i+1$, inducing actions on $\mathcal D$ and $\mathcal P$. The extended reflection groupoid is the action groupoid of $\mathcal D\rtimes W_{\mathrm{aff}}$ on $\mathcal P$, where $\mathcal D$ acts trivially. Its objects are the sequences $p$, and a morphism $p\to p'$ is an element $(d,w)$ with $wp=p'$; composition is multiplication in $\mathcal D\rtimes W_{\mathrm{aff}}$. Its generating arrows $P_i:p\to s_i p$ correspond to $(1,s_i)$, and the loops $D_a:p\to p$ correspond to sign vectors negative only at position $a$, with relations \eqref{eq:diagonal-sign-relations} and \eqref{eq:permutation-involutivity}--\eqref{eq:permutation-sign-conjugation}.
\end{definition*}

\subsection{Sign automorphisms and normalized permutation lifts}\label{subsec:sign-automorphisms-and-normalized-permutation-lifts}
For $d=(d_1,\ldots,d_N)\in\{1,-1\}^{N}$, define
\begin{equation}
\begin{aligned}
D_d(x^+_{j,r})&=d_jd_{j+1}^{-1}x^+_{j,r},&
D_d(x^-_{j,r})&=d_{j+1}d_j^{-1}x^-_{j,r},&
D_d(h_{j,r})&=h_{j,r}.
\end{aligned}
\label{eq:diagonal-sign-automorphisms}
\end{equation}
These are automorphisms because every defining relation is homogeneous for the root-lattice grading. Write $D_a$ for the sign vector which is $-1$ at position $a$ and $1$ elsewhere. Then
\begin{equation}
D_a^2=\mathrm{Id},\qquad D_aD_b=D_bD_a,\qquad D_1\cdots D_N=\mathrm{Id}.
\label{eq:diagonal-sign-relations}
\end{equation}
Let $T_i(p):Y_{\hbar}(\Pi(p))\to Y_{\hbar}(\Pi(s_i p))$ be the map of Theorem~\ref{thm:quantum-odd-reflections} when $p_i\neq p_{i+1}$. At an even vertex let $T_i(p)$ be the original braid automorphism
\begin{equation}
T_i(p)=\exp(\operatorname{ad}x^+_{i,0})\,
\exp\!\bigl(-\operatorname{ad}((-1)^{p_i}x^-_{i,0})\bigr)\,
\exp(\operatorname{ad}x^+_{i,0}).
\label{eq:even-lift-in-parity-groupoid}
\end{equation}
Here $s_i p=p$ at an even vertex. Local nilpotence and the inverse are proved in Theorem~\ref{thm:even-braid-automorphisms}. Introduce auxiliary lifts
\begin{equation}
P_i(p)=
\begin{cases}
T_i(p),&p_i\neq p_{i+1},\\
D_{i+1}T_i(p),&p_i=p_{i+1}.
\end{cases}
\label{eq:normalized-permutation-lifts}
\end{equation}
The diagonal map in this formula acts in the target algebra. We retain the notation $T_i$ for the original, generally noninvolutive, even braid lifts.

\subsection{Finite RTT identification for both parities}\label{subsec:finite-rtt-identification-for-both-parities}
Use the calibrated Drinfeld--RTT identification of Lemma~\ref{lem:finite-rtt-normalization}, with the inclusive prefix shifts $\gamma_j$ and the coefficients $A_a,B_a,Z_a$ defined there. Adjacent permutations of RTT indices send $t_{ab,r}$ to $t'_{s_i(a),s_i(b),r}$ and compose as actual permutations \cite[Lemma~2.24 and Corollary~2.38]{Tsymbaliuk}.

\begin{lemma}\label{lem:normalized-rtt-permutations}
The transported adjacent RTT permutation is the restriction of $P_i(p)$ to the retained finite Yangian. Its values on $x^+_{a,0},x^-_{a,0},h_{a,0},\widetilde h_{a,1}$ involve only the vertex indices $\{a,i\}$.
\end{lemma}

\begin{proof} For degree zero the resulting table, with primes in the target, is
\begin{align}
&x^+_{i,0}\longmapsto(-1)^{p_i+1}x'^{-}_{i,0},\quad
x^-_{i,0}\longmapsto(-1)^{p_i}x'^{+}_{i,0},\quad
h_{i,0}\longmapsto-h'_{i,0}.
\label{eq:permutation-degree-zero-vertex}\\
&x^+_{i-1,0}\longmapsto[x'^{+}_{i-1,0},x'^{+}_{i,0}],\quad
x^-_{i-1,0}\longmapsto(-1)^{p_i+1}[x'^{-}_{i,0},x'^{-}_{i-1,0}].
\label{eq:permutation-degree-zero-left}\\
&x^+_{i+1,0}\longmapsto[x'^{+}_{i,0},x'^{+}_{i+1,0}],\quad
x^-_{i+1,0}\longmapsto(-1)^{p_i}[x'^{-}_{i+1,0},x'^{-}_{i,0}].
\label{eq:permutation-degree-zero-right}
\end{align}
The neighboring Cartan elements map to $h'_{j,0}+h'_{i,0}$; all other degree-zero generators are fixed. The order-two Gauss formula gives
\begin{equation}
A_i\longmapsto A'_{i+1},\qquad A_{i+1}\longmapsto A'_i,
\label{eq:permutation-first-gauss-coefficients}
\end{equation}
\begin{equation}
\begin{aligned}
Z_i&\longmapsto Z'_{i+1}+(-1)^{p_i}\hbar x'^{-}_{i,0}x'^{+}_{i,0},\\
Z_{i+1}&\longmapsto Z'_i-(-1)^{p_i}\hbar x'^{+}_{i,0}x'^{-}_{i,0}.
\end{aligned}
\label{eq:permutation-second-gauss-coefficients}
\end{equation}
Every other $A_a,Z_a$ is fixed. Also
\[
\gamma'_i-\gamma_i=\frac{\hbar}{2}\bigl((-1)^{p_i}-(-1)^{p_{i+1}}\bigr),\qquad
\gamma'_j=\gamma_j\quad(j\neq i).
\]
At an odd vertex $x'^{+}_{i,0}x'^{-}_{i,0}+x'^{-}_{i,0}x'^{+}_{i,0}=h'_{i,0}$. Substitution into \eqref{eq:finite-rtt-shifted-cartan-embedding} gives precisely the established odd table
\begin{equation}
\begin{aligned}
P_i(\widetilde h_{i,1})&=-\widetilde h'_{i,1},\\
P_i(\widetilde h_{j,1})&=\widetilde h'_{j,1}+\widetilde h'_{i,1}
+\frac{\hbar b_{ji}}{2}\{x'^{+}_{i,0},x'^{-}_{i,0}\}
\quad(j=i\pm1).
\end{aligned}
\label{eq:odd-permutation-cartan-action}
\end{equation}
At an even vertex $x'^{+}_{i,0}x'^{-}_{i,0}-x'^{-}_{i,0}x'^{+}_{i,0}=h'_{i,0}$, and the same substitution gives
\begin{equation}
\begin{aligned}
P_i(\widetilde h_{i,1})&=-\widetilde h'_{i,1}-(-1)^{p_i}\hbar\{x'^{+}_{i,0},x'^{-}_{i,0}\},\\
P_i(\widetilde h_{j,1})&=\widetilde h'_{j,1}+\widetilde h'_{i,1}
+\frac{(-1)^{p_i}\hbar}{2}\{x'^{+}_{i,0},x'^{-}_{i,0}\}\quad(j=i\pm1).
\end{aligned}
\label{eq:even-permutation-cartan-action}
\end{equation}
The remaining $\widetilde h_{a,1}$ are fixed. These formulas are independent of the cut.

For completeness, the even table agrees with definition \eqref{eq:normalized-permutation-lifts}, including current degree one. The rescaled generators $x^+_{i,0},(-1)^{p_i}x^-_{i,0},(-1)^{p_i}h_{i,0}$ satisfy the $\mathfrak{sl}_2$ relations. The rank-one current relations include
\[
[\widetilde h_{i,1},x^+_{i,0}]=2(-1)^{p_i}x^+_{i,1},\qquad
[x^+_{i,0},x^+_{i,1}]=-(-1)^{p_i}\hbar(x^+_{i,0})^2.
\]
Expanding in the three finite exponentials gives
\[
\begin{aligned}
\exp(\operatorname{ad}x^+_{i,0})\widetilde h_{i,1}
&=\widetilde h_{i,1}-2(-1)^{p_i}x^+_{i,1}+\hbar(x^+_{i,0})^2,\\
\exp\!\bigl(-\operatorname{ad}((-1)^{p_i}x^-_{i,0})\bigr)
\exp(\operatorname{ad}x^+_{i,0})\widetilde h_{i,1}
&=-\widetilde h_{i,1}-2(-1)^{p_i}x^+_{i,1}+\hbar(x^+_{i,0})^2\\
&\quad+(-1)^{p_i}\hbar\{x^+_{i,0},h_{i,0}\}
-(-1)^{p_i}\hbar\{x^+_{i,0},x^-_{i,0}\},\\
T_i(\widetilde h_{i,1})
&=-\widetilde h_{i,1}-(-1)^{p_i}\hbar\{x^+_{i,0},x^-_{i,0}\}.
\end{aligned}
\]
All braces in these identities are ordinary anticommutators since the rank-one generators are even. Thus
$T_i(\widetilde h_{i,1})=-\widetilde h_{i,1}-(-1)^{p_i}\hbar(x^+_{i,0}x^-_{i,0}+x^-_{i,0}x^+_{i,0})$.
For a neighbor $j$, the element $\widetilde h_{j,1}+\widetilde h_{i,1}/2$ commutes with both $x^+_{i,0}$ and $(-1)^{p_i}x^-_{i,0}$, since $b_{ji}=-(-1)^{p_i}$ and $b_{ii}=2(-1)^{p_i}$. Its invariance gives the neighboring line of \eqref{eq:even-permutation-cartan-action}. The sign automorphism $D_{i+1}$ fixes these Cartan images, while changing the degree-zero braid table to \eqref{eq:permutation-degree-zero-vertex}--\eqref{eq:permutation-degree-zero-right}. This proves equality on an algebra generating set, and proves the lemma.\end{proof}

\subsection{Coxeter relations and the original braid lifts}\label{subsec:coxeter-relations-and-the-original-braid-lifts}
\begin{theorem}\label{thm:normalized-coxeter-relations}
The maps $P_i(p)$, for all cyclic vertices, satisfy the affine type $\widetilde A_{N-1}$ Coxeter relations as maps between the corresponding parity presentations:
\begin{equation}
P_i(s_i p)P_i(p)=\mathrm{Id},
\label{eq:permutation-involutivity}
\end{equation}
\begin{equation}
P_j(s_i p)P_i(p)=P_i(s_j p)P_j(p)\quad\text{if $i,j$ are distinct and nonadjacent},
\label{eq:permutation-distant-commutation}
\end{equation}
\begin{equation}
P_i(s_j s_i p)P_j(s_i p)P_i(p)=P_j(s_i s_j p)P_i(s_j p)P_j(p)\quad\text{if $i,j$ are adjacent}.
\label{eq:permutation-braid-relation}
\end{equation}
Moreover
\begin{equation}
P_i(p)D_a=D_{s_i(a)}P_i(p).
\label{eq:permutation-sign-conjugation}
\end{equation}
Every composition is read from right to left; each diagonal map acts in the algebra appropriate to its side of the equation.
\end{theorem}

\begin{proof} First consider a finite retained path. After its fixed Drinfeld--RTT identification, every $P_i$ is an actual transposition of matrix indices by Lemma~\ref{lem:normalized-rtt-permutations}. The intermediate identifications cancel in compositions. Hence the square, distant-commutation, and braid relations are exactly the corresponding permutation identities. Diagonal conjugation gives \eqref{eq:permutation-sign-conjugation}.

To prove a displayed relation in the affine Yangian, fix one generator among $x^+_{a,0}$, $x^-_{a,0}$, $h_{a,0}$, $\widetilde h_{a,1}$. For a relation involving $i,j$, all intermediate expressions involve only $\{a,i,j\}$, by the support statement of Lemma~\ref{lem:normalized-rtt-permutations}. Choose a cut vertex $k\notin\{a,i,j\}$, possible since $N\geq5$. All maps and all intermediate expressions belong to that finite path. The finite identity maps to the desired affine identity under the homomorphism from its finite Yangian to the target affine Yangian. For a square or diagonal-conjugation relation, omit $j$. The identities therefore hold on the indicated generating set, hence on the whole algebra. The argument includes $i=0$ and $j=0$.\end{proof}

\begin{theorem}\label{thm:original-braid-lifts}
The original lifts $T_i(p)$ satisfy distant commutation and the same object-indexed braid relation \eqref{eq:permutation-braid-relation}, with $P$ replaced by $T$, for every parity pattern. They also satisfy
\begin{equation}
T_i(s_i p)T_i(p)=
\begin{cases}
\mathrm{Id},&p_i\neq p_{i+1},\\
D_iD_{i+1},&p_i=p_{i+1},
\end{cases}
\label{eq:original-lift-square}
\end{equation}
and $T_i(p)D_a=D_{s_i(a)}T_i(p)$. In particular, the square of an even lift changes the sign of every adjacent root generator.
\end{theorem}

\begin{proof} Equation \eqref{eq:original-lift-square} follows from \eqref{eq:normalized-permutation-lifts}, \eqref{eq:permutation-involutivity}, and \eqref{eq:permutation-sign-conjugation}. The same formulas prove diagonal conjugation and distant commutation. For the braid relation we can again work on a finite path and then use the same cut argument. The maps $T_i$ there are induced by signed permutation matrices. On the two exchanged basis vectors, their action is
\[
v_i\longmapsto
\begin{cases}
-v_{i+1},&p_i=p_{i+1},\\
v_{i+1},&p_i\neq p_{i+1},
\end{cases}
\qquad v_{i+1}\longmapsto v_i.
\]
Consider three consecutive basis vectors, with initial parities $a,b,c\in\{0,1\}$. Both words $iji$ and $jij$ reverse their order. In both words the first vector acquires the sign $(-1)^{\delta_{ab}+\delta_{ac}}$, the second the sign $(-1)^{\delta_{bc}}$, and the third the sign $1$. Indeed, each pair crosses once; its left vector receives a minus sign exactly when the two parities coincide. Thus the signed permutation matrices, not merely their underlying permutations, are identical. This proves the mixed braid identity for all eight parity triples. The finite RTT realization and locality transfer it to the affine Yangian.\end{proof}

\begin{corollary}\label{cor:extended-groupoid-action}
The maps $P_i(p)$ and $D_a$ define an action of the \hyperref[def:extended-reflection-groupoid]{extended reflection groupoid} on the family $\Y(\Pi(p))$ by $\mathbb C[\hbar]$-algebra isomorphisms. Equivalently, this action can be expressed through the original lifts $T_i(p)$, their braid relations, the square formula \eqref{eq:original-lift-square}, and diagonal conjugation.
\end{corollary}

\begin{proof}
The sign automorphisms satisfy \eqref{eq:diagonal-sign-relations}, and Theorem~\ref{thm:normalized-coxeter-relations} verifies the remaining defining relations of the groupoid. Therefore these maps determine the asserted action. The invertible substitution \eqref{eq:normalized-permutation-lifts} expresses the generators $P_i$ through the original lifts $T_i$. Theorem~\ref{thm:original-braid-lifts} gives the corresponding braid, square, and diagonal-conjugation relations.
\end{proof}

\section{Completed coproducts and odd-reflection comparison}\label{sec:completed-coproducts-and-odd-reflection-comparison}
We first specify the degree-zero map and the topology. We then compare the independently prescribed root-sum assignments and prove that they are algebra homomorphisms for every cyclic parity sequence.

\subsection{The canonical degree-zero map}\label{subsec:the-canonical-degree-zero-map}
\begin{proposition}\label{prop:canonical-degree-zero-map}
For every cyclic parity sequence there is a unique degree-preserving $\mathbb C[\hbar]$-algebra homomorphism
\begin{equation}
\begin{aligned}
\jmath_\Pi &: U(\g')\otimes_{\mathbb C}\mathbb C[\hbar]\longrightarrow\Y(\Pi),\\
\jmath_\Pi(e_i)&=x^+_{i,0},\qquad \jmath_\Pi(f_i)=x^-_{i,0},\qquad \jmath_\Pi(h_i)=h_{i,0}.
\end{aligned}\label{eq:canonical-degree-zero-map}
\end{equation}
Here the standard loop grading is used on both sides.
\end{proposition}

\begin{proof}
For the distinguished sequence, the classical presentation in \cite[Proposition~2.4]{Ueda} gives this homomorphism. Every parity sequence with the prescribed numbers of even and odd entries can be reached from the distinguished sequence by successive swaps of unequal adjacent entries at the finite vertices $1,\ldots,N-1$. For one step, use the classical isomorphism $R_i$ of Proposition~\ref{prop:classical-odd-reflection} and the independently proved quantum isomorphism $T_i$ of Theorem~\ref{thm:quantum-odd-reflections}. Define
\[
\jmath_{\Pi'}=T_i\jmath_\Pi R_i^{-1}.
\]
The degree-zero formulas show that this homomorphism has the required target Chevalley images. At a finite vertex both $T_i$ and $R_i$ preserve the standard loop grading: their reflected root has degree zero, adjacent brackets retain the neighboring degree, and the quadratic Cartan corrections have degree zero. Induction gives existence for every sequence. The Chevalley generators generate $\g'$, so the prescribed images also give uniqueness and independence of the chosen sequence of swaps.
\end{proof}

In root sums we write affine elements for their images under $\jmath_\Pi$. In particular, $\jmath_\Pi(c)=\sum_i h_{i,0}$. The construction of this map precedes all coproduct arguments.

\subsection{The parity convention for the Casimir tensor}\label{subsec:the-parity-convention-for-the-casimir-tensor}
Choose a positive root system $\Delta_+$ of $\g$. Let $\{H_a\}$ and $\{H^a\}$ be dual bases of $\mathfrak h$. In each positive root space choose homogeneous bases $\{e_{\alpha,k}\}\subset\g_\alpha$ and $\{f_{\alpha,k}\}\subset\g_{-\alpha}$ satisfying
\begin{equation}
(e_{\alpha,k},f_{\alpha,l})=\delta_{kl},\qquad
(f_{\alpha,l},e_{\alpha,k})=(-1)^{p(\alpha)}\delta_{kl}.\label{eq:root-vector-duality}
\end{equation}
Root spaces of $\widehat{\mathfrak{sl}}(m|n)$ are parity homogeneous; imaginary root spaces are even. The canonical tensor is the formal sum
\begin{equation}
\Omega=\sum_a H_a\otimes H^a+\sum_{\alpha\in\Delta_+}\sum_k
\bigl(f_{\alpha,k}\otimes e_{\alpha,k}+(-1)^{p(\alpha)}e_{\alpha,k}\otimes f_{\alpha,k}\bigr).
\label{eq:full-casimir-tensor}
\end{equation}
The convention used here for the half Casimir tensor is
\begin{equation}
\Omega_{\Pi,+}=\sum_a H_a\otimes H^a+\sum_{\alpha\in\Delta_+}\sum_k f_{\alpha,k}\otimes e_{\alpha,k}.
\label{eq:half-casimir-tensor}
\end{equation}
The label $+$ records a choice of positive roots, not symmetry under the superpermutation. In particular, $\Omega_{\Pi,+}+\sigma(\Omega_{\Pi,+})=\Omega+\sum_a H_a\otimes H^a$, where $\sigma(a\otimes b)=(-1)^{p(a)p(b)}b\otimes a$.

\begin{lemma}\label{lem:casimir-invariance}
The tensor \eqref{eq:full-casimir-tensor} is independent of the chosen dual bases, satisfies $\sigma(\Omega)=\Omega$, and is invariant:
\[
[x\otimes1+1\otimes x,\Omega]=0\qquad (x\in\g).
\]
These statements are understood coefficientwise in root spaces, or as identities of operators wherever the sums act locally finitely.
\end{lemma}

\begin{proof}
For a homogeneous basis $\{u_a\}$ with right dual basis $\{u^a\}$, defined by $(u_a,u^b)=\delta_a^b$, the inverse tensor of the form is $\sum_a(-1)^{p(u_a)}u_a\otimes u^a$. This is the coevaluation tensor corresponding to the identity map under the identification with the dual induced by the form. It is therefore basis independent. Supersymmetry of the form gives invariance under $\sigma$, while invariance of the form gives invariance under the diagonal adjoint action. On a pair of opposite root spaces, the right duals of $e_{\alpha,k}$ and $f_{\alpha,k}$ are respectively $f_{\alpha,k}$ and $(-1)^{p(\alpha)}e_{\alpha,k}$; the displayed coevaluation tensor is precisely \eqref{eq:full-casimir-tensor}. The same verification applies coefficientwise to the affine root decomposition, whose root spaces are finite dimensional.
\end{proof}
The canonical tensor $\Omega$ and a quadratic Casimir element have different definitions. In finite type, applying multiplication to \eqref{eq:full-casimir-tensor} gives a quadratic element and permits normal ordering. In affine type, a normally ordered Casimir operator requires its own definition and a justification of the multiplication on its tensor domain.

\subsection{A degreewise completion and its scope}\label{subsec:a-degreewise-completion-and-its-scope}
The affine loop degree is different from the Drinfeld index $r$ and from the deformation degree in $\hbar$. For the standard affine labeling set
\begin{equation}
\deg h_{i,r}=0,\qquad \deg x^+_{i,r}=\delta_{i0},\qquad
\deg x^-_{i,r}=-\delta_{i0},\qquad \deg\hbar=0.\label{eq:standard-loop-grading}
\end{equation}
The defining relations are homogeneous for this grading. The tensor product has the sum grading.

\begin{definition}\label{def:degreewise-completion}
For a $\mathbb Z$-graded algebra $A=\bigoplus_{q\in\mathbb Z}A(q)$, put
\begin{equation}
F_L A(q)=\sum_{r>L}A(q-r)A(r),\qquad
\widehat A(q)=\varprojlim_L A(q)/F_L A(q),\qquad
\widehat A=\bigoplus_q\widehat A(q).\label{eq:degreewise-completion}
\end{equation}
Each component carries its linear inverse-limit topology. The canonical map $A\longrightarrow\widehat A$ is not presumed injective. For a graded Yangian $Y$, the notation $Y\widehat\otimes Y$ in this section means the degreewise completion of the algebra $Y\otimes Y$.
\end{definition}

The multiplication on homogeneous components extends to this completion. Indeed,
\[
A(a)F_L A(b)\subset F_L A(a+b),\qquad
F_L A(a)A(b)\subset F_{L+b} A(a+b),
\]
which give the required continuity and show that products of Cauchy sequences are Cauchy. A claim that the original algebra embeds into this completion requires a separate separation argument.

\begin{lemma}\label{lem:completion-functoriality}
A degree-preserving algebra homomorphism $f:A\longrightarrow\widehat B$ extends uniquely to a continuous degree-preserving algebra homomorphism $\widehat f:\widehat A\longrightarrow\widehat B$.
\end{lemma}

\begin{proof}
The image of $A(q-r)A(r)$ is contained in the closure of $B(q-r)B(r)$. Indeed, approximate the two homogeneous image factors by elements of $B$ and use continuity of multiplication. Consequently $f(F_L A(q))$ lies in the closure of $F_L B(q)$. Thus $f$ is continuous on every homogeneous component and extends to its completion, since $\widehat B(q)$ is complete and separated. The extensions are multiplicative by continuity and the density of the canonical images of the components of $A$.
\end{proof}
In particular, the insertion maps from $Y^{\otimes2}$ to $Y^{\otimes3}$ extend continuously. If a coproduct $\Delta:Y\to\widehat{Y^{\otimes2}}$ is a degree-preserving algebra homomorphism, define $\Delta\otimes\mathrm{Id}$ first on algebraic tensors by inserting $\Delta(a)$ in the first two positions and multiplying by $1\otimes1\otimes b$. This gives a graded algebra homomorphism $Y^{\otimes2}\to\widehat{Y^{\otimes3}}$; completed elements in the first two factors supercommute with the last factor by continuity. The lemma extends this map to $\widehat{Y^{\otimes2}}$. The construction for $\mathrm{Id}\otimes\Delta$ is identical. Both iterates have the common separated target $\widehat{Y^{\otimes3}}$.

In the root sums taking values in $Y$, the notation $e_{\alpha,k},f_{\alpha,k}$ denotes their images under the canonical map \eqref{eq:canonical-degree-zero-map}. For standard affine positive roots, write $\alpha=\bar\alpha+r\delta$. There are only finitely many roots, counted with multiplicity, of bounded nonnegative loop degree. For large $r$, a summand $f_{\alpha,k}\otimes e_{\alpha,k}$ belongs to $F_L(Y\otimes Y)(0)$ whenever $r>L$, since it factors as $(f_{\alpha,k}\otimes1)(1\otimes e_{\alpha,k})$. Consequently the root sums converge in this degreewise sense. The same construction for $U(\g)\otimes U(\g)$, with $\deg d=\deg c=0$, contains $\Omega_{\Pi,+}$ itself. Its Cartan term lies in the full affine tensor completion; the commutator with $h_i\otimes1$ belongs to the tensor completion associated with the derived algebra. For a different positive system, convergence follows by checking the same tail property or transporting the topology along an established isomorphism. The infinite root sum has fixed $\hbar$-degree and its convergence is supplied by the loop-degree topology.

The construction is the degreewise route used in \cite[Definition~4.1 and Theorem~4.3]{Ueda}. One may also first work on tensor products of suitable highest weight modules, as in \cite[Sections~4--5]{GNW}. Positive root vectors of sufficiently large degree annihilate a fixed vector, so the sums are locally finite there. Passing from module identities to an identity in an abstract completed algebra requires a separation or faithfulness argument.

\subsection{The root-sum assignment and coassociativity}\label{subsec:the-root-sum-assignment-and-coassociativity}
Use the canonical map $\jmath_\Pi$ from Proposition~\ref{prop:canonical-degree-zero-map}. Root vectors in $\Y(\Pi)$ mean their images under this map. Normalize the Cartan generators of the loop algebra so that $(h_i,-)=\alpha_i$ on the full Cartan; then $b_{ij}=(\alpha_i,\alpha_j)$.

Define an assignment on the indicated generators by
\begin{equation}
\Delta(\jmath_\Pi(a))=\jmath_\Pi(a)\otimes1+1\otimes\jmath_\Pi(a)
\qquad(a\in\g'),\label{eq:primitive-degree-zero-coproduct}
\end{equation}
\begin{equation}
\Delta(\widetilde h_{i,1})=\widetilde h_{i,1}\otimes1+1\otimes\widetilde h_{i,1}
-\hbar\sum_{\alpha\in\Delta_+}\sum_k(\alpha_i,\alpha)f_{\alpha,k}\otimes e_{\alpha,k}.
\label{eq:shifted-cartan-root-sum-coproduct}
\end{equation}
Equivalently, the second formula is
\[
\Delta(\widetilde h_{i,1})=\widetilde h_{i,1}\otimes1+1\otimes\widetilde h_{i,1}
+\hbar(\jmath_\Pi\widehat\otimes\jmath_\Pi)[h_i\otimes1,\Omega_{\Pi,+}].
\]
Although the Cartan part of $\Omega_{\Pi,+}$ uses the full affine algebra, it disappears from this commutator. The commutator therefore belongs to the tensor completion associated with the derived loop algebra, where the indicated extension of $\jmath_\Pi\otimes\jmath_\Pi$ is defined by degreewise functoriality.

\begin{lemma}\label{lem:cartan-root-coproduct}
If the assignment extends to an algebra homomorphism, then
\begin{equation}
\Delta(h_{i,1})=h_{i,1}\otimes1+1\otimes h_{i,1}+\hbar h_{i,0}\otimes h_{i,0}
-\hbar\sum_{\alpha\in\Delta_+}\sum_k(\alpha_i,\alpha)f_{\alpha,k}\otimes e_{\alpha,k}.
\label{eq:cartan-root-sum-coproduct}
\end{equation}
Conversely, \eqref{eq:cartan-root-sum-coproduct} and \eqref{eq:primitive-degree-zero-coproduct} imply \eqref{eq:shifted-cartan-root-sum-coproduct}.
\end{lemma}

\begin{proof}
The even element $h_{i,0}$ is primitive, hence $\Delta(h_{i,0}^2)=h_{i,0}^2\otimes1+2h_{i,0}\otimes h_{i,0}+1\otimes h_{i,0}^2$. Substitute the definition of $\widetilde h_{i,1}$.
\end{proof}
When $b_{ji}\ne0$, the relation $[\widetilde h_{j,1},x^\pm_{i,0}]=\pm b_{ji}x^\pm_{i,1}$ forces
\begin{equation}
\Delta(x^\pm_{i,1})=\pm b_{ji}^{-1}[\Delta(\widetilde h_{j,1}),\Delta(x^\pm_{i,0})].\label{eq:degree-one-root-coproduct}
\end{equation}
Selecting one such $j$ defines a candidate image; independence of the choice of $j$ follows only after compatibility with all the defining relations has been established.

For the distinguished presentation, Ueda proves that these formulas, in his two-parameter normalization, define a coassociative algebra homomorphism \cite[Theorem~4.3]{Ueda}. Its symmetric specialization is $\varepsilon_1=\varepsilon_2=\hbar/2$. Applying that result to another presentation requires an identification of the presentations, normalizations, and completions. It is not a consequence of the classical limit alone.

\begin{proposition}\label{prop:coproduct-coassociativity}
If \eqref{eq:primitive-degree-zero-coproduct}--\eqref{eq:shifted-cartan-root-sum-coproduct} extend to an algebra homomorphism
\[
\Delta:\Y(\Pi)\longrightarrow\Y(\Pi)\widehat\otimes\Y(\Pi),
\]
then it is degree preserving and coassociative in the standard degreewise completion of $\Y(\Pi)^{\otimes3}$. The augmentation $\epsilon$ sending all current generators to zero satisfies the completed counit identities
\[
(\epsilon\widehat\otimes\mathrm{Id})\Delta=\iota_{\Y(\Pi)}=(\mathrm{Id}\widehat\otimes\epsilon)\Delta,
\qquad \iota_{\Y(\Pi)}:\Y(\Pi)\longrightarrow\widehat{\Y(\Pi)}.
\]
When $\Y(\Pi)$ embeds into its completion, these are the counit identities under that identification.
\end{proposition}

\begin{proof} The displayed coproduct preserves the loop degree, so its two iterates exist and are continuous by Lemma~\ref{lem:completion-functoriality} and the construction immediately after that lemma. Coassociativity on $\jmath_\Pi(a)$, $a\in\g'$, is immediate from primitivity. Write the correction in \eqref{eq:shifted-cartan-root-sum-coproduct} as $C_i=\sum_\lambda u_\lambda\otimes v_\lambda$,
with its scalar coefficients incorporated into $u_\lambda$. Both tensor factors are primitive; the coefficients belong to the base ring. Therefore both iterated coproducts of $\widetilde h_{i,1}$ are
\[
\widetilde h_{i,1}\otimes1\otimes1+1\otimes\widetilde h_{i,1}\otimes1
+1\otimes1\otimes\widetilde h_{i,1}+(C_i)_{12}+(C_i)_{13}+(C_i)_{23}.
\]
Here $(C_i)_{13}=\sum_\lambda u_\lambda\otimes1\otimes v_\lambda$; inserting units introduces no parity signs. The convergence of the root sums and continuity of the iterated maps justify these operations. The iterates are algebra homomorphisms and agree on generators, so they agree on $\Y(\Pi)$. The tensor extensions of the augmentation exist by Lemma~\ref{lem:completion-functoriality}. Applying the augmentation to either factor kills every term of $C_i$ and verifies both completed counit identities on generators, hence on $\Y(\Pi)$.\end{proof}

The hypothesis that $\Delta$ is an algebra homomorphism is substantive. To prove it from a minimal presentation, one must verify the mixed degree-zero/degree-one relations, the relations between equal-sign degree-one generators, commutativity of all required Cartan generators, and the complete Serre relations of the presentation. The zero-degree relations follow from the enveloping-algebra coproduct. The remaining checks, particularly $[\Delta(h_{i,1}),\Delta(h_{j,1})]=0$, cannot be replaced by the coassociativity calculation above.

\begin{proposition}\label{prop:coproduct-transport}
Let $T:Y\to Y'$ be an even algebra isomorphism. Suppose it induces continuous isomorphisms on the tensor completions used for $\Delta$ and its iterates. Then
\[
\Delta^{\mathrm{tr}}=(T\widehat\otimes T)\circ\Delta\circ T^{-1}
\]
is a coassociative algebra homomorphism whenever $\Delta$ is one. The transported counit is $\epsilon^{\mathrm{tr}}=\epsilon\circ T^{-1}$. If $Y'$ also has an independently defined coproduct $\Delta'$, then $T$ intertwines the two coproducts if and only if $\Delta'=\Delta^{\mathrm{tr}}$.
\end{proposition}

\begin{proof}
Conjugate the homomorphism, coassociativity, and counit identities by the relevant tensor powers of $T$. The last equivalence follows directly from the definition. All compositions exist by the continuity hypotheses.
\end{proof}
In particular, transport supplies a coproduct on an isomorphic presentation, but does not by itself identify that coproduct with the root-sum formula for its new positive roots. This final identification must be checked on an algebra generating set.
The criterion proves coassociativity once multiplicativity is known. The reflection comparison below proves multiplicativity for every parity sequence.

\subsection{Odd reflection and the half Casimir tensor}\label{subsec:odd-reflection-and-the-half-casimir-tensor}
The change in a positive-root coproduct under one odd reflection is finite. Its quadratic correction can therefore be computed exactly, without rearranging an infinite sum. This computation both identifies the transported coproduct and specifies the normalization of the degree-one Cartan generators that the comparison requires.

Use the normalized degree-zero reflection to identify the source and target copies of the centrally extended loop superalgebra. For the odd reflected source root $\alpha_i$, use normalized root vectors $e_{\alpha_i},f_{\alpha_i}$ with $(e_{\alpha_i},f_{\alpha_i})=1$ and $h_{\alpha_i}=[e_{\alpha_i},f_{\alpha_i}]$. These denote elements of the common degree-zero algebra. Put
\[
C_\Pi(h)=[h\otimes1,\Omega_{\Pi,+}],\qquad
\partial_0(a)=\Delta_0(a)-a\otimes1-1\otimes a,
\]
where $\Delta_0$ is the primitive coproduct on its enveloping algebra.

\begin{lemma}\label{lem:odd-reflection-casimir-change}
For the reflected positive system $\Pi'$,
\begin{equation}
\Omega_{\Pi',+}-\Omega_{\Pi,+}=-f_{\alpha_i}\otimes e_{\alpha_i}-e_{\alpha_i}\otimes f_{\alpha_i},\label{eq:half-casimir-odd-reflection}
\end{equation}
\begin{equation}
C_{\Pi'}(h)-C_\Pi(h)=(h,\alpha_i)(f_{\alpha_i}\otimes e_{\alpha_i}-e_{\alpha_i}\otimes f_{\alpha_i}).\label{eq:casimir-commutator-change}
\end{equation}
Consequently the even quadratic element
\begin{equation}
q_h=\frac{(h,\alpha_i)}2(e_{\alpha_i}f_{\alpha_i}-f_{\alpha_i}e_{\alpha_i})
=(h,\alpha_i)\left(e_{\alpha_i}f_{\alpha_i}-\frac12h_{\alpha_i}\right)\label{eq:quadratic-cartan-correction}
\end{equation}
satisfies
\begin{equation}
\partial_0(q_h)=C_\Pi(h)-C_{\Pi'}(h).\label{eq:reduced-coproduct-correction}
\end{equation}
\end{lemma}

\begin{proof}
An odd reflection replaces the single positive root $\alpha_i$ by $-\alpha_i$; every other positive root, including every positive imaginary root, is unchanged. For the cyclic systems this can also be read directly from the periodic sequence of chosen basis vectors: interchanging two adjacent entries changes the order of precisely that pair. A normalized dual pair for $-\alpha_i$ has the form $u f_{\alpha_i},-u^{-1}e_{\alpha_i}$, with $u\in\mathbb C^\times$, because $(f_{\alpha_i},e_{\alpha_i})=-1$. Its contribution is therefore $(-u^{-1}e_{\alpha_i})\otimes(uf_{\alpha_i})=-e_{\alpha_i}\otimes f_{\alpha_i}$. The Cartan tensor is unchanged, which proves \eqref{eq:half-casimir-odd-reflection}. Since $[h,e_{\alpha_i}]=(h,\alpha_i)e_{\alpha_i}$ and $[h,f_{\alpha_i}]=-(h,\alpha_i)f_{\alpha_i}$, commuting with $h\otimes1$ gives \eqref{eq:casimir-commutator-change}.

Both $e_{\alpha_i}$ and $f_{\alpha_i}$ are odd and primitive. The tensor-product sign rule gives
\[
\begin{aligned}
\partial_0(e_{\alpha_i}f_{\alpha_i})&=e_{\alpha_i}\otimes f_{\alpha_i}-f_{\alpha_i}\otimes e_{\alpha_i},\\
\partial_0(f_{\alpha_i}e_{\alpha_i})&=f_{\alpha_i}\otimes e_{\alpha_i}-e_{\alpha_i}\otimes f_{\alpha_i}.
\end{aligned}
\]
Substitution proves \eqref{eq:reduced-coproduct-correction}.
\end{proof}
For $h\in\mathfrak h'$, let $t_h$ denote the corresponding linear combination of the source generators $\widetilde h_{j,1}$, and let $t'_h$ denote the linear combination of the target generators $\widetilde h'_{j,1}$ associated with the same degree-zero Cartan element. The normalized odd lift has the Cartan formula
\begin{equation}
T_i(t_h)=t'_h+\hbar q_h.\label{eq:reflection-cartan-correction}
\end{equation}
Here $q_h$ is viewed in the target Yangian through the canonical map and the common degree-zero identification. Equivalently, in the primed Yangian generators,
\begin{equation}
T_i(t_h)=t'_h+\frac\hbar2(h,\alpha'_i)
\bigl(x'{}^-_{i,0}x'{}^+_{i,0}-x'{}^+_{i,0}x'{}^-_{i,0}\bigr).\label{eq:reflection-cartan-correction-target}
\end{equation}
Indeed $\alpha'_i=-\alpha_i$, and the classical generators satisfy $e_{\alpha_i}=-(-1)^{p_i}f'_i$ and $f_{\alpha_i}=(-1)^{p_i}e'_i$. In particular,
\begin{equation}
\begin{aligned}
T_i(\widetilde h_{i,1})&=-\widetilde h'_{i,1},\\
T_i(\widetilde h_{j,1})&=\widetilde h'_{j,1}+\widetilde h'_{i,1}
-\hbar b_{ji}\left(x'{}^-_{i,0}x'{}^+_{i,0}-\frac12h'_{i,0}\right)
\quad(j=i-1,i+1),
\end{aligned}\label{eq:reflected-shifted-cartan-values}
\end{equation}
where $b_{ji}$ is the source coefficient; the other generators $\widetilde h_{j,1}$ are unchanged. The difference $e_{\alpha_i}f_{\alpha_i}-f_{\alpha_i}e_{\alpha_i}$ in \eqref{eq:quadratic-cartan-correction} is the ordinary commutator of the two odd elements, or equivalently their super-anticommutator. Replacing it by their supercommutator would give the primitive element $h_{\alpha_i}$ and would not correct the root tensor.

\subsection{Comparison and propagation of the coproduct}\label{subsec:comparison-and-propagation-of-the-coproduct}
\begin{theorem}\label{thm:coproduct-reflection-comparison}
Suppose that the source root-sum assignment is an algebra homomorphism $\Delta_\Pi:\Y(\Pi)\longrightarrow\Y(\Pi)\widehat\otimes\Y(\Pi)$. Suppose that $T_i:\Y(\Pi)\longrightarrow\Y(\Pi')$ is an algebra isomorphism with the normalized degree-zero restriction and
the Cartan formula \eqref{eq:reflection-cartan-correction}, and that it and its inverse extend to the chosen tensor completions. These are the normalizations of the odd lift in Theorem~\ref{thm:quantum-odd-reflections}.

Then the transported coproduct
\[
\Delta^{\mathrm{tr}}_{\Pi'}=(T_i\widehat\otimes T_i)\Delta_\Pi T_i^{-1}
\]
agrees on generators with the independently specified target positive-root assignment:
\begin{equation}
\Delta^{\mathrm{tr}}_{\Pi'}(\jmath_{\Pi'}(a))
=\jmath_{\Pi'}(a)\otimes1+1\otimes\jmath_{\Pi'}(a)
\quad(a\in\g'),\label{eq:transported-degree-zero-coproduct}
\end{equation}
\begin{equation}
\Delta^{\mathrm{tr}}_{\Pi'}(t'_h)=t'_h\otimes1+1\otimes t'_h+\hbar C_{\Pi'}(h).\label{eq:transported-cartan-coproduct}
\end{equation}
In particular, the target root-sum assignment extends to an algebra homomorphism and
\[
\Delta_{\Pi'}T_i=(T_i\widehat\otimes T_i)\Delta_\Pi.
\]
If the source coproduct is coassociative, then so is this target coproduct.
\end{theorem}

\begin{proof}
The transported map is an algebra homomorphism and is primitive on the common degree-zero algebra. Its restriction to the enveloping algebra generated there is therefore $\Delta_0$. Using $t'_h=T_i(t_h)-\hbar q_h$ gives
\[
\begin{aligned}
\Delta^{\mathrm{tr}}_{\Pi'}(t'_h)
&=(T_i\widehat\otimes T_i)\Delta_\Pi(t_h)-\hbar\Delta_0(q_h)\\
&=t'_h\otimes1+1\otimes t'_h+\hbar\bigl(C_\Pi(h)-\partial_0(q_h)\bigr)\\
&=t'_h\otimes1+1\otimes t'_h+\hbar C_{\Pi'}(h),
\end{aligned}
\]
where the last equality is \eqref{eq:reduced-coproduct-correction}. The degree-zero generators and the $t'_h$ generate the target algebra: for each $j$, choose $k$ with $b'_{kj}\ne0$ and use
\[
x'{}^\pm_{j,1}=\pm(b'_{kj})^{-1}[\widetilde h'_{k,1},x'{}^\pm_{j,0}].
\]
Thus \eqref{eq:transported-degree-zero-coproduct}--\eqref{eq:transported-cartan-coproduct} identify the entire transported algebra map with the prescribed root-sum assignment. Transport preserves coassociativity.
\end{proof}

\subsection{The coproduct theorem for arbitrary cyclic parity}\label{subsec:the-coproduct-theorem-for-arbitrary-cyclic-parity}
\begin{theorem}\label{thm:coproduct-all-parities}
Let $m,n\ge2$, $m\ne n$, and use the supercommutator and superanticommutator conventions of the full and minimal presentations. For every parity sequence with $m$ even and $n$ odd entries there is a canonical degree-preserving algebra homomorphism
\[
\jmath_\Pi:U(\g')\otimes_{\mathbb C}\mathbb C[\hbar]\longrightarrow\Y(\Pi)
\]
with the prescribed Chevalley-generator images. The assignments
\begin{equation}
\Delta_\Pi(\jmath_\Pi(a))=\jmath_\Pi(a)\otimes1+1\otimes\jmath_\Pi(a)
\qquad(a\in\g'),\label{eq:all-parities-degree-zero-coproduct}
\end{equation}
\begin{equation}
\Delta_\Pi(\widetilde h_{i,1})=\widetilde h_{i,1}\otimes1+1\otimes\widetilde h_{i,1}
-\hbar\sum_{\beta\in\Delta_{\Pi,+}}\sum_k\beta(h_i)f_{\beta,k}\otimes e_{\beta,k}\label{eq:all-parities-shifted-cartan-coproduct}
\end{equation}
define a unique coassociative algebra homomorphism
\[
\Delta_\Pi:Y_\hbar(\Pi)\longrightarrow\widehat{Y_\hbar(\Pi)^{\otimes2}}.
\]
Here the root vectors mean their images under $\jmath_\Pi$. In particular, this assignment preserves every relation of the minimal presentation, including Cartan--Cartan commutativity and all additional Serre relations. The result is independent of the sequence of finite odd reflections used in its construction. The usual augmentation satisfies the completed counit identities.
\end{theorem}

\begin{proof}
The canonical degree-zero map is supplied by Proposition~\ref{prop:canonical-degree-zero-map}. For the distinguished coproduct we use the polynomial-parameter form of Ueda's verification. Set $\varepsilon_1=\varepsilon_2=\hbar/2$ at the start of that verification. The identities of \cite[Lemma~4.7, Corollary~4.21, Proposition~4.26, and Corollary~4.57]{Ueda}, followed by the relation checks (4.65)--(4.70) there, are calculations over $\mathbb C[\hbar]$: their coefficients are polynomial in the parameters, and all divisions are by nonzero numerical Cartan or root pairings and integers. The root-vector inductions and Casimir contractions are polynomial over the same coefficient ring. The same loop-degree cutoffs justify the root sums over this ring. Consequently the proof of \cite[Theorem~4.3]{Ueda} establishes the displayed assignment as a $\mathbb C[\hbar]$-algebra homomorphism in the specified degreewise completion. The full/minimal presentation equivalence justifies the finite list of relations. Coassociativity follows from Proposition~\ref{prop:coproduct-coassociativity}.

Every parity string with the specified numbers of entries can be obtained from the distinguished string by successively interchanging adjacent unequal entries in positions $1,\ldots,m+n$. These are reflections at the finite vertices $1,\ldots,m+n-1$. The normalized local formulas for each such reflection preserve the standard loop grading: the reflected root vectors have loop degree zero, their adjacent brackets have the degree of the adjacent generator, and their quadratic Cartan corrections have degree zero. This remains true for a neighboring affine generator, whose degree is $1$ or $-1$.

Apply Theorem~\ref{thm:coproduct-reflection-comparison} at every step. It constructs a coassociative algebra homomorphism and proves the positive-root formula at the new simple system. The positive roots of positive loop degree have the same tail throughout these finite reflections, so all root sums converge in the already specified degreewise completion.

Every defining relation vanishes under the resulting map because it is a conjugate of an algebra homomorphism by established algebra isomorphisms. This reasoning applies to all relations simultaneously, including the degree-one Cartan relation.

Finally, the degree-zero generators together with $\widetilde h_{i,1}$ generate the minimal presentation. Indeed, choose $j$ with $b_{ji}\ne0$ and write
\[
x^\pm_{i,1}=\pm b_{ji}^{-1}[\widetilde h_{j,1},x^\pm_{i,0}].
\]
Thus \eqref{eq:all-parities-shifted-cartan-coproduct} and primitivity determine the homomorphism uniquely. Every reflection path gives these same formulas and hence the same map. The augmentation annihilates every root vector and every $\widetilde h_{i,1}$, so its completed counit identities follow on this generating set and therefore on the algebra.
\end{proof}
For the unmodified degree-one Cartan generators, the theorem gives
\[
\Delta_\Pi(h_{i,1})=h_{i,1}\otimes1+1\otimes h_{i,1}+\hbar h_{i,0}\otimes h_{i,0}
 -\hbar\sum_{\beta\in\Delta_{\Pi,+}}\sum_k\beta(h_i)f_{\beta,k}\otimes e_{\beta,k}.
\]
The coproducts of the degree-one root generators follow from the displayed
commutator reconstruction and are independent of the chosen neighboring index.

\par\medskip
An odd reflection at vertex $0$ is graded from the source standard loop
grading to the target physical grading: the latter assigns $-1$ to the
reflected positive simple root, $1$ to its two neighbors, and $0$ to the
others. The established isomorphism and its inverse extend to these
compatible degreewise completions. The target root sum converges there,
since its tail agrees with the source tail by
\eqref{eq:half-casimir-odd-reflection}, so
Theorem~\ref{thm:coproduct-reflection-comparison} applies.

\par\medskip
Restriction along the reflection isomorphisms gives equivalences of module
categories. The normalized odd lifts make them compatible with tensor
products whenever the completed coproducts act on the modules concerned.
At vertex $0$ the compatible target is the physical-loop completion.

\par\medskip
\textbf{PBW.} For the distinguished parity sequence $\Pi$, the companion manuscript \cite[Theorem~1.1]{PBW} proves a PBW theorem for the specialization $\Y(\Pi)/(\hbar-1)\Y(\Pi)$. With respect to the current filtration, in which the generators $x^\pm_{i,r}$ and $h_{i,r}$ have weight $r$, its associated graded algebra is
\[
\operatorname{gr}_{\mathrm{cur}}\bigl(\Y(\Pi)/(\hbar-1)\Y(\Pi)\bigr)
\cong U\!\left(\operatorname{uce}\bigl(\mathfrak{sl}(m|n)\otimes\mathbb C[z^{\pm1},u]\bigr)\right).
\]
Here $z$ is the affine loop variable and $u$ records the classical current degree. Choose any totally ordered basis of this universal central extension, homogeneous in parity and current degree, and arbitrary but fixed filtered lifts of its elements. The distinct ordered monomials in these lifts, including $1$, form a $\mathbb C$-basis of the specialization. Even basis elements admit arbitrary nonnegative exponents, while odd basis elements have exponent zero or one. The filtered reflection isomorphisms of Theorem~\ref{thm:quantum-odd-reflections} transport these bases to every cyclic parity sequence at $\hbar=1$.

\par\medskip
\textbf{Acknowledgement.} This work was supported by the Russian Science
Foundation, grant 26-11-00115, and carried out at the Center of Fundamental
Mathematics, Moscow Institute of Physics and Technology.

\end{document}